\documentclass[11pt,reqno]{amsart}
\usepackage{dsfont}

\usepackage{enumitem}
\usepackage{bbm}
\usepackage{times}
\usepackage{tabularx}
\usepackage{amsbsy,amssymb,amsmath,amsthm,amscd,amsfonts,latexsym}
\usepackage{mathtools}
\usepackage{tikz}
\usetikzlibrary{arrows,decorations.markings}
\usetikzlibrary{matrix}
\usetikzlibrary{graphs}
\usetikzlibrary{backgrounds}
\usepackage{setspace}     \spacing{1}
\usepackage{color}
\usepackage{mathrsfs}
\usepackage{float}
\usepackage{indentfirst}
\usepackage{graphicx}
\usepackage[colorlinks=true,linkcolor=blue,citecolor=blue]{hyperref} %checking references
\newtheorem{thm}{Theorem}
\newtheorem{prop}{Proposition}%[section]
\newtheorem{lem}[prop]{Lemma}
\newtheorem{cor}[prop]{Corollary}

\theoremstyle{definition}

\newtheorem{df}[prop]{Definition} %[subsection]

\newtheorem{conj}[prop]{Conjecture}

\theoremstyle{remark}

\newtheorem{rmk}[prop]{Remark} %[subsection]
\newtheorem{example}[prop]{Example}

\newcommand{\F}{{\mathbb{F}}}
\newcommand{\R}{{\mathbb{R}}}
\newcommand{\Z}{{\mathbb{Z}}}
\newcommand{\C}{{\mathbb{C}}}

\newcommand{\N}{{\mathbb{N}}}

\newcommand{\bD}{{\mathbb{D}}}

\newcommand{\M}{\mathcal{M}}

\newcommand{\bH}{\mathbb{H}}
\newcommand{\ham}{\mathcal{H}am}

\newcommand{\cL}{\mathcal{O}}

\newcommand{\cls}{c_{LS}}

\newcommand{\overbar}{\overline}

\newcommand{\al}{\alpha}

\newcommand{\ga}{\gamma}

\newcommand{\T}{\mathbb{T}}

\newcommand{\ox}{\overline{x}}

\newcommand{\cA}{\mathcal{A}}

\newcommand{\cD}{\mathcal{D}}

\newcommand{\cH}{\mathcal{H}}

\newcommand{\cJ}{\mathcal{J}}

\newcommand{\cO}{\mathcal{O}}

\newcommand{\cP}{\mathcal{P}}

\newcommand{\cT}{\mathcal{T}}
\newcommand{\cM}{\mathcal{M}}

\newcommand{\brat}[1]{{\left< #1 \right>}}
\newcommand{\tens}{\otimes}
\newcommand{\ahl}{\mathcal{A}_{H,L}}

\DeclareMathOperator{\aut}{\mathrm{Aut}}
\DeclareMathOperator{\im}{\mathrm{Im}}
\DeclareMathOperator{\spec}{\mathrm{Spec}}

\DeclareMathOperator{\crit}{{\mathrm{Crit}}}
\DeclareMathOperator{\ind}{{\mathrm{ind}}}
\DeclareMathOperator{\virdim}{\mathrm{vir-dim}}

\numberwithin{equation}{section}

\title[Lagrangian intersections and a conjecture of Arnol'd]{Lagrangian intersections and a conjecture of Arnol'd}

\author{Wenmin Gong}

\address{Laboratory of Mathematics and Complex Systems (Ministry
of Education), School of Mathematical Sciences, Beijing Normal University,
	Beijing, 100875, China}

\email{ wmgong@bnu.edu.cn}
\begin{document}
\maketitle
%\centerline{Preliminary version}

\begin{abstract}
We prove a degenerate homological  Arnol'd conjecture on Lagrangian intersections beyond the case studied by A. Floer and H. Hofer via a new version of Lagrangian Ljusternik--Schnirelman theory. We introduce the notion of (Lagrangian) fundamental quantum factorizations and use them to give some uniform lower bounds of the numbers of Lagrangian intersections  for some classical examples including Clifford tori in complex projective spaces.  Additionally, we use the Lagrangian Ljusternik-Schnirelman theory to study the size of the intersection of a monotone Lagrangian with its image of a Hamiltonian diffeomorphism.

\end{abstract}

\maketitle

\section{Introduction}\label{sec:1}

\noindent

In the present paper we are interested in Lagrangian
intersections in symplectic manifolds. Specifically, given a compact Lagrangian submanifold $L$ of a symplectic manifold $(M,\omega)$ and a Hamiltonian diffeomorphism $\varphi$, one would like to bound from below the number of intersection points of $L$ and $\varphi(L)$. In its strongest form, the Lagrangian version of the Arnol'd conjecture, known for about 60 years,   asserts that (possibly under additional conditions)
the number of intersection points of $L$ and $\varphi(L)$ is bounded from below by the minimal possible number of critical points that a smooth function on $L$ may have, see~\cite{Ar0,Ar}. This version of the conjecture is still open, but Floer theory~\cite{Fl1} solves a weaker version of the conjecture, i.e., whenever $L$ and $\varphi(L)$ intersect transversely, the number should be at least the sum of the Betti numbers of the homology group $H_*(L,\F)$ with coefficients in any field $\F$; this weaker version is known as the nondegenerate homological Arnol'd conjecture.

A somewhat stronger version of conjecture (though still weaker than the one mentioned above) is that, without the hypothesis that $L$ and $\varphi(L)$ intersect transversely, the number of intersection points must be at least the cup-length of $L$, which is defined as
\begin{eqnarray}
cl(L):=\max\big\{k+1:&\exists\; a_i\in H_{d_i}(L,\F),\; d_i<n,\;i=1,\ldots,k\notag\\
&\hbox{such that}\;a_1\cap\ldots \cap a_k\neq0\big\}.\notag
\end{eqnarray}

Even though this degenerate version of homological Arnol'd conjecture is generally open, it has been settled in certain special cases.  Hofer~\cite{Ho} proved it for the zero section of a cotangent bundle $T^*L$, see also~\cite{LS} by Laudenbach and Sikorav.  Under the assumption that $\pi_2(M,L)=0$ or $\omega(\pi_2(M,L))=0$, Floer~\cite{Fl2} and Hofer~\cite{Ho1} proved it. Liu~\cite{Liu} proved it under the condition that the Hofer norm of a Hamiltonian diffeomorphism $\|\varphi\|_{Hofer}$ is strictly less than the minimal area of a non-constant $J$-holomorphic disk on $L$ or a non-constant $J$-holomorphic sphere in $M$ (this condition was first used in the work of Chekanov~\cite{Ch1,Ch2}).

\begin{conj}[{Degenerate homological Arnol'd conjecture (cf.~\cite[Chapter~11]{MS0})}]\label{conj:fixpt}
Let $(M,\omega)$ be a closed symplectic manifold. Then for any Hamiltonian diffeomorphism $\varphi$ of $M$,  the number of the fixed points of $\varphi$ is at least the $\F$-cup-length of $M$.
\end{conj}

An affirmative answer to the following Lagrangian version of homological Arnol'd conjecture, which can be seen as a variant of the \textbf{degenerate homological Arnol'd-Givental conjecture} \cite[Conjecture~11.3.1]{MS0}, would imply Conjecture~\ref{conj:fixpt} for monotone symplectic manifolds.
\begin{conj}\label{conj:intersection}
Assume that $L$ is a closed non-displaceable monotone Lagrangian submanifold of a closed symplectic manifold $M$. For any Hamiltonian diffeomorphism $\varphi$ of $M$, the number of intersection points of $\varphi(L)$ and $L$ is at least the $\F$-cup-length of $L$.
\end{conj}

On page 577 of \cite{Fl3}, Floer states that ``Our methods do not yield cup-length estimates for $\sharp \hbox{Fix}(\varphi)$ for general monotone manifolds. (We tend to believe that there are more than technical reasons for this)." This statement still seems to hold true today. All known approaches to Conjecture~\ref{conj:fixpt} encounter conceptual difficulties in distinguishing contributions from the same periodic orbit with different cappings (a cap of $1$-periodic orbit $\gamma$ is a topological $2$-disk with boundary $\gamma$, cf. Section~\ref{sec:hfh}). This difficulty is naturally present in Conjecture~\ref{conj:intersection}, but it is even more involved now.

Lagrangian Floer homology (see ~\cite{FOOO2} for the general definition) is a powerful tool for solving the  Arnol'd conjecture in the non-degenerate case since the seminal work of Floer~\cite{Fl1,Fl3}. However, it seems to the author that there are relatively few results  in literature about Conjecture~\ref{conj:intersection} so far. There is still room for further development in this direction. A possible candidate to do this is to apply Ljusternik--Schnirelman theory as we have already seen its usefulness in previous works ~\cite{Ho1,Fl2,Fl3,LO,Sc2,EG1,EG2,GG,GG2,Go}. This is the very reason to extend Ljusternik-Schnirelman theory to the monotone Lagrangian submanifolds in the present paper.  Very recently, 
there are new developments in this direction, see~\cite{HP,AAC}. As a main application of this theory, we show that

\begin{thm}\label{cor:RPn}  Conjecture~\ref{conj:intersection} is true for $(M,L)=(\mathbb{C}P^n,\;\R P^n),\;(\mathbb{C}P^n\times( \mathbb{C}P^n)^-,\;\Delta_{\mathbb{C}P^n}), \;$
$(Gr(2,2n+2),\;\bH P^n)$, $(Q^n,\;S^n)$. 
	
\end{thm}

Note that each of these four Lagrangians $L$ is the fixed-point set of an
antisymplectic involution on $M$ (cf.~\cite[Section~3]{KS}), and thus the degenerate homological Arnol'd-Givental conjecture holds for them.
Theorem~\ref{cor:RPn} recovers the well-known result by Givental~\cite{Gi} that $\sharp(\R P^n\cap\varphi(\R P^n))\geq n+1$ for all Hamiltonian diffeomorphisms $\varphi$ of $\mathbb{C}P^n$, see Chang and Jiang~\cite{CJ} for a different proof.  
We also mention that this result was generalized by Lu~\cite{Lu} to the weighted projective spaces $(\mathbb{C}P^n(\mathbf{q}),\;\R P^n(\mathbf{q}))$ with odd weights $\mathbf{q}=(q_1,\ldots,q_{n+1})\in\N^{n+1}$.

Our next main theorem is:
\begin{thm}\label{thm:2pt}
Let $\T_{Clif}^n$ be the Clifford torus in $\mathbb{C}P^n$.
For any Hamiltonian diffeomorphism $\varphi$ of $\mathbb{C}P^n$, $\varphi(\T_{Clif}^n)$ has at least two intersection points with $\T_{Clif}^n$.
\end{thm}

Theorem~\ref{thm:2pt} strengthens a well-known result that $\T_{Clif}^n\subset \mathbb{C}P^n$ is not displaceable by a Hamiltonian isotopy, see~\cite{BEP,Cho}.

\medskip

\subsection{Notations and conventions}\label{subsec:notation}
Throughout this paper, we assume that $(M,\omega)$ is a connected and tame symplectic
manifold (see~\cite{ALP}). Such manifolds include closed symplectic manifolds, open manifolds
that are symplectically convex at infinity, as well as products of such. We denote
by $\cJ$ the space of $\omega$-compatible almost complex structures such that
$(M,g_{J})$ is geometrically bounded, where for $J\in\cJ$, $g_{J}
(\cdot,\cdot)=\omega(\cdot,J\cdot)$ is the associated Riemannian metric. Denote  by $\cH$ the space of functions $H\in C^\infty([0,1]\times M,\R)$ such that 
$H(t,x)$ is constant with respect to $x$ outside of some compact set of $M$ for all $t\in[0,1]$.  We denote by
$\{\varphi_H^t\}_{t\in[0,1]}$ the Hamiltonian isotopy of $H$ obtained by
integrating the time-dependent vector field $X_{H_t}$, where $H_t=H(t,\cdot)$ and
$X_{H_t}$ is determined uniquely by $-dH_t=\omega(X_{H_t},\cdot)$. We denote
by $\ham(M,\omega)$ the group of all Hamiltonian
diffeomorphisms generated by elements of $\cH$.

Recall that a Lagrangian submanifold $L\subseteq M$ is called \textit{monotone} if the two homomorphisms
$$\omega:\pi_2(M,L)\to\R,\quad \mu: \pi_2(M,L)\to\Z,$$
given by pairing with the symplectic form and the Maslov class, respectively, satisfy
$$\omega=\kappa_L\mu\quad \hbox{for some positive constant } \kappa_L.$$
We define the \textit{minimal Maslov number} of $L$ to be the integer
$$N_L=\min\big\{\mu(A)|A\in\pi_2(M,L),\;\mu(A)>0\big\}.$$

Throughout this paper we assume that all $L$ are connected and closed monotone submanifolds.  In this case it is known that
$(M,\omega)$ is (spherically) \textit{monotone}, which means that
$$\omega(A)=2\kappa_Lc_1(A),\quad \forall A\in\pi_2(M),$$
where $c_1=c_1(TM,\omega)$ is the first Chern class of $M$. We denote by $C_M$ the minimal positive Chern number of $M$
$$C_M=\min\big\{c_1(A)|A\in\pi_2(M),\;c_1(A)>0\big\}.$$
If the homomorphisms $\omega$ and $\mu$ vanish on $\pi_2(M,\omega)$, i.e.,
$$\omega(A)=\mu(A)=0,\quad A\in \pi_2(M,L),$$
we call $L$ \textit{weakly exact}, and in this case we have $N_L=\infty$. Similarly, if
$$\omega(A)=c_1(A)=0,\quad \forall A\in \pi_2(M),$$
we call $(M,\omega)$ \textit{symplectically aspherical}.
%In what follows, we shall call a Lagrangian submanifold (resp. symplectic manifold) \textit{monotone} if it is weakly monotone and not weakly exact (resp. symplectically aspherical).
For any monotone Lagrangian $L$ of $(M,\omega)$  we have that $N_L$ divides $2C_M$. In what follows we denote by $A_L=\kappa_LN_L$ the minimal positive generator of $\omega(\pi_2(M,L))$, and let $A_L=0$ in the weakly exact case.
Since the Maslov numbers are multiples of $N_L$, for simplicity we also use the notation
$$\overbar{\mu}=\frac{1}{N_L}\mu: \pi_2(M,L)\to\Z$$

In this paper we work with $\Z_2$-coefficient unless otherwise specified.
Throughout we assume that the minimal Maslov number of $L$ is at least two, i.e., $N_L\ge 2$. The reason for   this assumption is that as the main tools in this paper, Lagrangian Floer homology~$HF(L)$ and Lagrangian quantum homology $QH(L)$ with $\Z_2$-coefficient are well-defined, see~\cite{Oh1,BC}. 
For a monotone Lagrangian submanifold $L$, we denote by $\Lambda=\Z_2[t^{-1},t]]$ the Novikov field of semi-infinite Laurent series in $t$ graded by $\deg t=-N_L$. Each element of $\Lambda$ is a semi-infinite sum $\sum_{k\geq 0}a_kt^{l_k}$, where $a_k\in\{0,1\}$ and the sequence $\{l_k\}_{k=0}^\infty\subseteq\Z$ diverges to $+\infty$ as $k$ tends to infinity.

 We define a valuation map $\nu:\Lambda\to \Z\cup \{-\infty\}$ as
\begin{equation}\label{e:val}
\nu\bigg(\sum\limits_{k\geq 0} a_kt^{l_k}\bigg)=\max\big\{-l_k|a_k\neq 0\big\}.
\end{equation}
Similarly, let $\Gamma=\Z_2[s^{-1},s]]$ be the Novikov field of semi-infinite Laurent polynomials in $s$, where the degree of $s$ is $-2C_M$. Sending $s$ to $t^{2C_M/N_L}$ induces a degree preserving ring homomorphism $\Gamma\to \Lambda$; this will be relevant when comparing the quantum homologies of $M$ and the pair $(M,L)$.

Notice that the construction of Lagrangian Floer  homology for closed monotone Lagrangians can be extended to the setting of tame symplectic manifolds. The
analytical difficulties are present already in the closed case, and have been worked
out in \cite{Fl1,Oh1,BC,LZ}, etc. The generalization from closed manifolds to tame manifolds is straightforward and follows from the monotonicity
theorem in \cite{Gr} since all the Hamiltonian systems under consideration are compactly supported. 

Recall that~\cite{BC,BC2,BC3} if there exists an isomorphism $HF(L)\cong H(L,\Z_2)\tens\Lambda$ as vector spaces over $\Lambda$,  then $L$ is said to be \emph{wide}; if $HF(L)=0$ then $L$ is said to be \emph{narrow}. Examples of wide Lagrangians are $\R P^n$, the Clifford torus in $\mathbb{C}P^n$, and weakly exact Lagrangian submanifolds. Note that all of those Lagrangians appearing in Theorems~\ref{cor:RPn} and \ref{thm:2pt} are wide, see~\cite{BC,KS}.

Throughout this paper,  $QH(L)$ denotes the Lagrangian quantum homology of $L$ (cf. Section~\ref{sec:lqh}). We denote by $\circ: QH(L)\tens QH(L)\longrightarrow QH(L)$ the Lagrangian quantum product  and $\bullet: QH(M,\Lambda)\tens QH(L)\longrightarrow QH(L)$ the module structure. Let $I_\omega$ and $\nu$ be the valuation maps on $QH(M,\Lambda)$ and $QH_*(L)$ induced by (\ref{e:val}) (see the definitions~(\ref{e:val0}) and (\ref{e:val1})), respectively. Denote by $\ell:QH(L)\times \cH\to\R$ the Lagrangian spectral invariant. For the above definitions, we refer to Sections~\ref{sec:qh}, \ref{sec:lqh} and \ref{sec:lsi} for the details.

\subsection{Main results}\label{subsec:Main}
For a monotone non-narrow Lagrangian $L$ with $N_L\geq 2$, we denote  
\[
\widehat{QH}(L)=\bigoplus_{r\in\Z}\bigoplus_{n-(r+1)N_L<p<n-rN_L}QH_p(L).
\]
This vector space contains the subspace $H_{n-N_L<*<n}(L)\otimes \Lambda$  
which has a \emph{canonical} embedding into $QH(L)$, see Section~\ref{subsec:emb}. 

\begin{thm}[Lagrangian Ljusternik--Schnirelman inequality~I]\label{thm:lls}
	Let $L^n\subset (M^{2n},\omega)$ be a monotone wide Lagrangian  with $N_L\geq 2$.
	Let $H\in \cH$, and
	let $\alpha,\beta$ be two non-zero elements of $QH_*(L)$.  Then we have
	\[
	\ell(\al\circ \beta, H)\leq \ell(\beta, H)+A_L\nu(\al).
	\]
Assume that $\alpha\in\widehat{QH}(L)$, and the intersections of $L$ and $\varphi_H(L)$ are isolated. Then
\[
\ell(\al\circ \beta, H)< \ell(\beta, H)+A_L\nu(\al).
\]
Furthermore, if there exists a nonzero class $\al=\sum_i x_i \tens\lambda_i \in \widehat{QH}(L)$ with each homogeneous class $x_i\in H_{n-N_L<*<n}(L,\Z_2)$  and each $\lambda_i\in\Lambda$
such that
\[
\ell(\al\circ \beta, H)=\ell(\beta, H)+A_L\nu(\al),
\]
then $L\cap \varphi_H(L)$ is homologically non-trivial in $L$.
\end{thm}

%\begin{rmk}The condition that $L$ is wide in Theorem~\ref{thm:lls} is not so strange as it looks at first glance since all known monotone Lagrangians are either wide or narrow. It was conjectured by Biran and Cornea~\cite{BC} that any monotone Lagrangian submanifold is either narrow or wide. Note that $L\cap\varphi(L)=\emptyset$ implies that $L$ is narrow. So if the \textit{wide--narrow} conjecture is right, the wide condition we put on $L$ is not much restrictive.\end{rmk}

In the above statement a subset $S$ of a topological space $X$ is called \textit{homologically non-trivial} in $X$ if for every open neighborhood $V$ of $S$ the map $i_*:H_k(V)\to H_k(X)$ induced by the inclusion $i:V\hookrightarrow X$ is non-zero for some $k>0$.

\begin{rmk}
	The third assertion in Theorem~\ref{thm:lls} is motivated by a result of Howard in~\cite{How} . Recently, Buhovsky, Humili\`{e}re and Seyfadini~\cite{BHS1} adapted Howard's method and gave a generalisation of the  Arnol'd conjecture  to non-smooth settings on a closed symplectically aspherical manifold. For the closely related work, we refer to~\cite{BHS2,Ka}.
\end{rmk}

\begin{thm}[Lagrangian Ljusternik--Schnirelman inequality~II]\label{thm:ll}
	Suppose that $L^n$ is  a monotone non-narrow Lagrangian in a closed symplectic manifold 
	$(M^{2n},\omega)$ with $N_L\geq 2$. Let $H\in \cH$, $ a\in QH(M,\Lambda)$ and let $\al$ be a nonzero class in $QH(L)$. Then we have
$$\ell(a\bullet\al,H)\leq\ell(\al,H)+I_\omega(a).$$
If there exists a nonzero class $a=\sum_i x_i \tens_\Gamma\lambda_i\in QH(M,\Lambda)$ with each homogeneous class $x_i\in H_{*<2n}(M,\Z_2)$, then either 
$$\ell(a\bullet\al,H)<\ell(\al,H)+I_\omega(a).$$
or $L\cap \varphi_H(L)$ is homologically non-trivial in $M$.

\end{thm}

In the following, we will give an example of $(M,L,H)$ such that the \textbf{strict} inequality in Theorem~\ref{thm:ll} fails and $L\cap \varphi_H(L)$ is homologically non-trivial in $M$.
\begin{example}
	We denote by $h=[\C P^1]\in H_2(\C P^2,\Z_2)$ the class of a hyperplane so that $h$ generates $QH(\C P^2,\Lambda)$ (cf. Example~\ref{ex:cp}). Let $\alpha_k=[\R P^k]\in H_k(\R P^2,\Z_2), k=0,1,2$ be the generator in $QH_k(\R P^2)$. It can be shown that $h\bullet \alpha_2=\alpha_0$, see~\cite[Lemma~6.1.1]{BC}. 
	For any constant function $f=c$ on $\C P^2$, by the normalization property (LS1) of $\ell$ (cf. Section~\ref{sec:lsi}), we have
	$\ell(\alpha_2,f)=\ell(\alpha_0,f)=c$.  This, together with $I_\omega(h)=0$,  implies
	$$\ell(h\bullet\al_2,f)=\ell(\al_2,f)+I_\omega(h).$$
 Clearly, $ \R P^2\cap \varphi_f(\R P^2)=\R P^2$. It is a standard fact that $[\R P^2]$ defines a non-trivial class in $H_2(\C P^2,\Z_2)$, see \cite{Hat}. So $\R P^2\cap \varphi_f(\R P^2)$ is homologically non-trivial in $\C P^2$. 
	
\end{example}

For weakly exact or monotone Lagrangians with $N_L\geq 2$ in tame symplectic manifolds,  Lagrangian spectral invariants can be constructed by using Floer theory  as in~\cite{LZ,Le,MVZ} with minor modifications. 
Since there are no holomorphic disks with boundaries on a weakly exact Lagrangian, the Lagrangian quantum product coincides with the intersection product. In this case by convention $N_L=\infty$, $A_L=0$, and $L$ is obviously wide since $HF(L)\cong H(L,\Z_2)$ by Floer~\cite{Fl1}.  Thus, Theorem~\ref{thm:lls} implies the following.

\begin{cor}\label{cor:weak}
	Let $L$ be a closed weakly exact Lagrangian submanifold of a tame symplectic manifold $(M,\omega)$, and let $\varphi_H\in\ham(M,\omega)$. If the total number of spectral invariants of $(L,H)$ is smaller than $\Z_2$-cup-length of $L$, then $L\cap \varphi_H(L)$ is homologically non-trivial in $L$. 
\end{cor}

%\subsection{Applications}
In particular, for $M=T^*L$, Corollary~\ref{cor:weak}  recovers the corresponding result implicitly contained in the work of Buhovsky, Humili\`{e}re and Seyfadini~\cite{BHS2}. Beyond the weakly exact case, the Lagrangian Ljusternik--Schnirelman inequalities~I and II
result in the following two corollaries.

\begin{cor}\label{cor:num}
	Let $L^n\subset (M^{2n},\omega)$ be a monotone wide Lagrangian  with minimal Maslov number $N_L\geq 2$.
	If there exist $k$ nonzero Lagrangian quantum homology classes $\al_i\in \widehat{QH}(L)$, $i=1,\ldots,k$ with $\nu(\al_i)<0$ and $0\neq\beta\in QH(L)$ such that $\al_1\circ\ldots\circ\al_k\circ\beta\neq0$, then for any $H\in \cH$,
	the total number of Lagrangian spectral invariants of the pair $(L,H)$ is at least $k+1$.
	
\end{cor}

\begin{cor}
	Let $L^n$ be  a  monotone non-narrow Lagrangian in a closed symplectic manifold $ (M^{2n},\omega)$ with  $N_L\geq 2$. If  there exist $k$ nonzero quantum homology classes $a_i\in \widehat{QH}(M,\Lambda)$, $i=1,\ldots,k$ with $I_\omega(a_i)<0$ and $0\neq\beta\in QH(L)$ which satisfy $a_1\bullet\ldots\bullet a_k\bullet\beta\neq0$, then for any $H\in \cH$,  the total number of Lagrangian spectral invariants of the pair $(L,H)$ is at least $k+1$.
\end{cor}

\subsubsection{A Chekanov-type result}

%To obtain some estimates of the number of intersection points in $\varphi_H(L)\cap L$, we need further information about the Hofer distance~\cite{Ho2} between two Lagrangians.
%Following Hofer and Chekanov~\cite{Ho2,Ch3}, define a pseudo-norm of a Hamiltonian $H\in \cH$$$\|H\|_L=\int^1_0\big(\max_{x\in L}H(t,x)-\min_{x\in L}H(t,x)\big)dt.$$

 %and correspondingly $\ham(M,\omega)$ the group of all Hamiltonian diffeomorphisms  generated by elements of $\cH$ if $M$ is compact.
 Denote by  $$\cL(L)=\{\varphi(L)|\varphi\in\ham(M,\omega)\}$$  the orbit of $L$ under the Hamiltonian diffeomorphism group $\ham(M,\omega)$. 
 %For $L_1,L_2\in\cL(L)$ we define $\cL(L)\times \cL(L)\to\R$ by setting
%$$d_H(L_1,L_2)=\inf_{H\in\cH}\bigg\{\int^1_0{\rm osc}_M H_t dt\bigg|\varphi^1_H(L_1)=L_2\bigg\},$$
%where ${\rm osc}_MH=\max_M H-\max_M H$. It turns out that this function $d_H$ is a genuine metric on $\cL(L)$ which is invariant under the action of $\ham(M,\omega)$ in our setting that $(M,\omega)$ is geometrically bounded and $L$ is a closed Lagrangian submanifold, see~\cite{Oh2,Ch3}. We call $d_H$ the \textit{Lagrangian Hofer metric} of $\cL(L)$.

%Following~\cite{Ho2,Ch3} the Lagrangian Hofer distance on $\cL(L)$ is defined as$$d_{H}(L,\varphi(L)):=\inf_{H\in\cH}\big\{\|H\|_L\big|\varphi \;\hbox{is generated by}\;H\big\}.$$

Recall that~\cite{Le,KS} if $QH(L)\neq 0$, then the \textit{Lagrangian spectral pseudo-norm} of the pair $(L,H)$ with $H\in\cH$ is defined by $$\ga(L,H)=\ell(L,H)+\ell(L,\overbar{H}),$$ where $\overbar{H}(t,x)=-H(-t,x)$. 
%This quantity is non-negative, see~(LS13) in Section~\ref{sec:lsi}.

For any $L'\in\cL(L)$ we define
$$\ga(L,L'):=\inf_{H\in\cH}\big\{\ga(L,H)\big|\varphi_H^1(L)=L'\big\}.$$
This pseudo-metric, whenever $L$ is closed and monotone with $N_L\geq 2$ and $QH(L)\neq0$, is non-degenerate and invariant under the action of $\ham(M,\omega)$, see~\cite{KS}. 

	The following \textbf{Chekanov-type result} on Lagrangian intersections partially verifies a more general format in~\cite[Section~8]{KS}. Denote by $d_H$  the Lagrangian Hofer distance on $\cL(L)$ (cf. \cite{Ho2,Ch3}).
It is well known that $\gamma \leq d_H$, see, e.g., \cite{KS}. In view of this inequality, our result strengthens Liu's \cite{Liu} in the monotone case.

\begin{thm}\label{thm: Arnol'dC}
Let $L^n\subset (M^{2n},\omega)$ be a monotone wide Lagrangian with $N_L\geq 2$. Suppose that  the singular homology $H(L,\Z_2)$ is generated as a ring by $H_{\geq n+1-N_L}(L,\Z_2)$ where the product is given by the intersection product. If $\varphi\in\ham(M,\omega)$ with $\ga(L,\varphi(L))<A_L$, then
$$\sharp \big(L\cap\varphi(L)\big)\geq cl(L).$$
\end{thm}

We remark here that under the assumption that $L$ and $\varphi(L)$ intersect transversely, a sharpened Chekanov-type result has been already established in~\cite[Theorem~E]{KS}.

It is a standard fact that the Clifford torus $$\T_{Clif}^n=\{[z_0:\cdots:z_n]\in\mathbb{C}P^n||z_0|=\cdots=|z_n|\}$$ is monotone with $N_{\T_{Clif}^n}=2$, see~\cite{Cho,BC}.
Let $t_1,\ldots,t_n$ be a basis of $H_{n-1}(\T_{Clif}^n,\Z_2)$ dual to the basis $[c_1],\ldots,[c_n]\in H_1(\T_{Clif}^n,\Z_2)$ with respect to the intersection product. Clearly, $H_*(\T_{Clif}^n,\Z_2)$ is generated by $t_1,\ldots,t_n$ and the fundamental class $[\T_{Clif}^n]$.
Therefore, by Theorem~\ref{thm: Arnol'dC}, any Hamiltonian diffeomorphism $\varphi$ of $\mathbb{C}P^n$ with
$\ga(\T_{Clif}^n,\varphi(\T_{Clif}^n))<A_{\T_{Clif}^n}$ satisfies
$$\sharp \big(\T_{Clif}^n\cap\varphi(\T_{Clif}^n)\big)\geq n+1.$$

The author expects that  $\ga(\T_{Clif}^n,\varphi(\T_{Clif}^n))<A_{\T_{Clif}^n}$ holds for any $\varphi\in\ham(\mathbb{C}P^n,\omega_{FS})$. In other words, Conjecture~\ref{conj:intersection} holds for $(\mathbb{C}P^n,\T_{Clif}^n)$,  but the author was unable to prove this.

It is shown in~\cite[Theorem~3.1]{Se} that $n+1$ is the maximal possible value of $N_L$ for a monotone Lagrangian $L$ in projective space $\mathbb{C}P^n$. The following examples of wide Lagrangians $L$ with $N_L>\dim L$ are studied by Biran and Cornea, see~\cite[Section~6]{BC}.
\begin{example}
If $L$ is a Lagrangian submanifold of $\mathbb{C}P^n$ with $2H_1(L,\Z)=0$, then $L$ is monotone and wide with $N_L=n+1$, see~\cite[Corollary~1.2.11]{BC}.
If $L$ is a Lagrangian submanifold of the quadric
$Q:=\{z_0^2+\cdots+z_n^2=z_{n+1}^2\}\subset\mathbb{C}P^n$ with $H_1(L;\Z)=0$, then $L$ is monotone and wide with $N_L=2n$, see~\cite[Lemma~6.3.2]{BC}. Therefore, by Theorem~\ref{thm: Arnol'dC}, for these two kinds of Lagrangians $L$, if $\ga(L,\varphi(L))<A_L$ then the number of intersections of $\varphi(L)$ with $L$ is at least the $\Z_2$-cup-length of $L$.
\end{example}

Other interesting examples which satisfy the conditions of Theorem~\ref{thm: Arnol'dC} include:
$$(M,L)=(\mathbb{C}P^n,\;\R P^n), \quad (\mathbb{C}P^n\times( \mathbb{C}P^n)^-,\;\Delta_{\mathbb{C}P^n}), \quad(Q^n,\;S^n),\quad(Gr(2,2n+2),\;\bH P^n),$$
where $\Delta_{\mathbb{C}P^n}$ is the diagonal in the product symplectic manifold $(\mathbb{C}P^n\times( \mathbb{C}P^n)^-,\omega_{FS}\oplus(-\omega_{FS}))$, $Q^n\subset \mathbb{C}P^{n+1}$ ($n>1$) is the complex quadric as described above, $S^n=Q^n\cap\R P^{n+1}$ is the natural monotone Lagrangian sphere in $Q^n$,
and $\bH P^n$ ($n\geq1$)  is the  quaternionic projective space in
the complex Grassmannian $Gr(2,2n+2)$. For $(\mathbb{C}P^n,\;\R P^n)$, it has been known that $\ga(\R P^n,\varphi(\R P^n))\leq \langle [\omega_{FS}], \C P^n \rangle$, see~\cite{EP0,KS}.
Moreover, the above four examples satisfy $\ga(L,\varphi(L))<A_L$ for any $\varphi\in\ham(M,\omega)$. This is a non-trivial fact which
was discovered by Kislev and Shelukhin by using an averaging method, see~\cite[Theorem~G]{KS} for more precise estimates of various Lagrangian spectral norms. Now we see that this phenomenon is closely related to Conjecture~\ref{conj:intersection}. 
By Theorem~\ref{thm: Arnol'dC}, we have $\sharp(L\cap \varphi(L))\geq cl(L)$ in these four cases, see Theorem~\ref{cor:RPn}. For $L=\R P^n, \Delta_{\mathbb{C}P^n}, \bH P^n$, it is known that $cl(L)=n+1$ (cf.~\cite[Section~3.2]{Hat}).

 \subsubsection{Uniform lower bounds for Lagrangian intersections }
We introduce the definitions of (Lagrangian) fundamental quantum factorizations to give some uniform lower bounds of the number of Lagrangian intersections for some typical examples. Putting forward this notion is inspired by the quantum cup-length proposed by Schwarz~\cite{Sc2}. In this subsection we assume that $L$ is a closed monotone Lagrangian with $N_L\geq 2$ in a closed symplectic manifold $(M,\omega)$. 
 \begin{df}\label{def:fqf}
We say that $M$ has a \textit{ fundamental quantum factorization} (FQF)  \textit{of length $k$} if there exist $u_1,\ldots,u_k\in H_{*<2n}(M,\Z_2)$ and $\tau\in\Z$ such that
$$t^\tau[M]= u_1*u_2*\cdots *u_k\quad \hbox{in}\; QH(M,\Lambda),$$
where the integer $\tau$ is called
the \textit{order} of FQF.  Clearly, $\tau> 0$ by degree reasons.
\end{df}

The following  symplectic manifolds satisfy the FQF property;  for relevant calculations of the quantum homology we refer to McDuff and Salamon~\cite{MS}.

\begin{example}
	(1) The complex projective spaces $\mathbb{C}P^n$ and complex Grassmannians; (2) The quadric $Q\subset\mathbb{C}P^n$; (3) $\mathbb{C}P^{n}\times\mathbb{C}P^{m_1}\times\cdots\times\mathbb{C}P^{m_r}$ with $m_1+1,\ldots,m_r+1$ divisible by $n+1$ and equally normalized symplectic structures, see, e.g., ~\cite{GG,BC}.  
	
\end{example}

\begin{thm}\label{thm:two} Let $L^n\subset (M^{2n},\omega)$ be a monotone wide Lagrangian with $N_L\geq 2$.
Suppose that $M$ has a FQF of length $k$ with order $\tau$. Let  $\varphi\in \cH am(M,\omega)$.  If the intersections of $L$ and $\varphi(L)$ are isolated, then the number of $L\cap\varphi(L)$ is at least $\lceil k/\tau\rceil$. Here $\lceil\cdot\rceil$ denotes the smallest integer that is greater or equal to the given number.
\end{thm}

\begin{rmk}
	Note that if $W^{2n}$ and $L_0$ satisfy the hypotheses of Theorem~\ref{thm:two}, and $P^{2l}$ is symplectically aspherical and contains a weakly exact Lagrangian $L_1$, then $W\times P$ and $L=L_0\times L_1$ also satisfy these hypotheses. Indeed,  the direct sum property of Malov index~\cite[Theorem~2.3.7]{MS0} implies that $N_L=N_{L_0}$.  Moreover, if $u_1,\ldots,u_k\in H_{*<2n}(W,\Z_2)$ satisfy the equality in Definition~\ref{def:fqf} with order $\tau$, then by the K\"{u}nneth formula of quantum homology \cite[Exercise~11.1.15]{MS}, $u_i\otimes[P] \in H_{*<2n+2l}(W\times P,\Z_2)$, $i=1\ldots k$, also satisfy this equality with order $\tau$. 
	\end{rmk}	

\begin{example}\label{ex:cp}
Let $L\subset \mathbb{C}P^n$ be a closed and monotone Lagrangian submanifold with minimal Maslov index $N_L$. We
denote by $h=[\mathbb{C}P^{n-1}]\in H_{2n-2}(\mathbb{C}P^n,\Z_2)$ the class of a hyperplane in the quantum homology $QH(\mathbb{C}P^n,\Lambda)$. It is shown in~\cite[Example~11.1.10]{MS} that
\[
h^{*k}=
\begin{cases}h^{\cap k},\quad  \ \ & 0\leq k\leq n,\\
[\mathbb{C}P^n]t^{(2n+2)/N_L}, \quad \ \ & k=n+1.
\end{cases}
\]
If $2H_1(L,\Z)=0$, then for such Lagrangian $L$, $\mathbb{C}P^n$   has a FQF of length $n+1$ with order $2$. By Theorem~\ref{thm:two}, for any $\varphi\in\ham(\mathbb{C}P^n,\omega_{FS})$  the number of $L\cap\varphi(L)$ is at least $\lceil\frac{n+1}{2}\rceil$.
\end{example}

\begin{example}
Let
$Q\subset\mathbb{C}P^{n+1}$ be the quadric as before.  Let $L\subset Q$ be a monotone Lagrangian with $H_1(L;\Z)=0$.
For $n=2k$, let $a,b\in H_{2k}(Q;\Z)$ be two classes of complex $k$-dimensional planes lying in $Q$ that generate $H_n(Q,\Z)$, see~\cite{GH}.
Let $u\in H_{2n}(Q;\Z)$ be the fundamental class and $p\in H_0(Q,\Z)$ the class of a point.
%, $h\in H_{2n-2}(Q,\Z_2)$  the class of a hyperplane section,
 The quantum product on $QH(Q,\Lambda)$ satisfies: (i) if $k$ is odd, then $a*b=p$, $a*a=b*b=ut$; (ii) if $k$ is even, then $a*a=b*b=p$, $a*b=ut$, see~\cite[Section 6.3.1]{BC}. These computations show that for such $L$, the quadric $Q$ has a FQF of length $2$ with order $1$. Therefore, for any $\varphi\in\ham(Q,\omega)$  the number of points in $L\cap\varphi(L)$ is at least $2$.
\end{example}

%\begin{cor}\label{cor:quadric}If $L$ is a Lagrangian in $ \mathbb{C}P^n$ with $2H_1(L;\Z)=0$, or a Lagrangian in the quadric $Q^{2k}\subset \mathbb{C}P^{2k+1}$ ($k\geq 1$) with $H_1(L;\Z)=0$, then for any $\varphi\in\ham$  the number of $L\cap\varphi(L)$ is at least $\lceil\frac{n+1}{2}\rceil$ or $2$, respectively.\end{cor}

Here we notice that a Lagrangian $L\subset \mathbb{C}P^n$ with $2H_1(L;\Z)=0$ has minimal Maslov number $n+1$ and is shown to be homotopy equivalent to $\R P^n$, see~\cite{KoS}. It was conjectured by Biran and Cornea~\cite{BC} that such a Lagrangian $L$ must be Hamiltonian isotopic to $\R P^n$. If the conjecture is true, then for $(\mathbb{C}P^n,L)$ the uniform lower bound given by Theorem~\ref{cor:RPn} is obviously better than the one by Theorem~\ref{thm:two}.
We also note that for the natural monotone Lagrangian sphere $L=S^{2k}\subset Q^{2k}$, the uniform lower bounds given by Theorem~\ref{thm:two} and Theorem~\ref{cor:RPn} are the same since $cl(S^{2k})=2$.

Similarly to Definition~\ref{def:fqf}, using the Lagrangian quantum product on $QH(L)$, one can propose the following concept.

\begin{df}% Suppose that $H_*(L,\Z_2)$ is generated as a ring by $H_{\geq n+1-N_L}(L,\Z_2)$.
		Let $L^n\subset (M^{2n},\omega)$ be a monotone wide Lagrangian with $N_L\geq 2$. We say that $L$ has a \textit{Lagrangian fundamental quantum factorization} (LFQF)  \textit{of length $l$} if there exist $v_1,\ldots,v_l\in H_{n-N_L<*<n}(L,\Z_2)$ and $\nu\in\Z$ such that
		$$t^\nu[L]=  v_1\circ v_2\circ \cdots \circ v_l\quad \hbox{in}\; QH(L),$$
		where $\nu$ is called
		the \textit{order} of LFQF.  By degree reasons  again we have $\nu>0$.
\end{df}
Correspondingly, one can prove the following.
\begin{thm}\label{thm:more}%Suppose that $H_*(L,\Z_2)$ is generated as a ring by $H_{\geq n+1-N_L}(L,\Z_2)$.
	Let $L^n\subset (M^{2n},\omega)$ be a monotone wide Lagrangian with $N_L\geq 2$, and let  $\varphi\in \ham(M,\omega)$. If $L$ has a LFQF of length $l$ with order $\nu$, then the number of $L\cap\varphi(L)$ is at least $\lceil l/\nu \rceil$.
\end{thm}

It is well known that the only Lagrangian submanifold on the sphere $S^2$ which is monotone is the ``equator", by which we mean the embedding circle separating the sphere into two disks of equal areas. It is a standard fact that such an equator $L$ is \emph{non-displaceable} in the sense that for every Hamiltonian diffeomorphism $\varphi$ of $S^2$ we have $L\cap\varphi(L)\neq 0$. Moreover, we have $\sharp (L\cap\varphi(L))\geq 2$. This obvious fact can be confirmed in many ways, for instance, by using Theorem~\ref{cor:RPn}. Here we provide another proof from the Lagrangian Ljusternik-Schnirelman theory.

In fact, for $L=\text{equator} \subset \mathbb{C}P^1=S^2$, an easy calculation shows that $[pt]\circ [pt]=[L]t$, where $[pt]$ is the point class for $L$, and $\deg t=-2$. So $L$ has a LFQF of length $2$ with order $1$, and thus Theorem~\ref{thm:more} implies that $L\cap\varphi(L)$ has at least two elements whatever $\varphi\in\ham(S^2)$ is.

More generally, let $t_1,\ldots,t_n$ be a basis of $H_{n-1}(\T_{Clif}^n,\Z_2)$ dual to the basis $[c_1],\ldots,[c_n]\in H_1(\T_{Clif}^n,\Z_2)$ as before. It can be shown that for $i\neq j$, $t_i\circ t_j+t_j\circ t_i=[\T_{Clif}^n]t$, and for every $i$, $t_i\circ t_i=[\T_{Clif}^n]t$, see~\cite{BC,Cho2}. So $\T_{Clif}^n$ has LFQF of length $2$ with order $1$, and thus Theorem~\ref{thm:2pt} holds true.

\begin{example}
Let $M=\mathbb{C}P^1\times \mathbb{C}P^1\cong S^2\times S^2$ with the monotone symplectic form $\omega_{FS}\times\omega_{FS}$, and let $L=\R P^1\times\R P^1\cong \T^2$ be the Clifford torus  in $S^2\times S^2$ (which is monotone with $N_L=2$). We set $a=[\R P^1\times pt], b=[pt\times \R P^1]\in H_1(L,\Z_2)$. It can be shown that $a\circ b=b\circ a=[pt]$ and $a\circ a=b\circ b=[L]t$, see~\cite[Example~1.2]{Ha}. So $\R P^1\times\R P^1$ has a LFQF of length $2$ with order $1$, and hence $\sharp(\R P^1\times\R P^1\cap \varphi(\R P^1\times\R P^1))\geq 2$ for all $\varphi\in\ham(\mathbb{C}P^1\times \mathbb{C}P^1,\omega_{FS}\times\omega_{FS})$.

\end{example}

\subsection{Organization of the paper} In Section~\ref{sec:pre} we sum up preliminaries from Floer theory including Lagrangian and Hamiltonian Floer homologies, and from quantum homology including Lagrangian quantum homology of a Lagrangian submanifold and quantum homology of the ambient manifold. In Section~\ref{sec:spectinv} we list the basic properties of Hamiltonian and Lagrangian spectral invariants. In particular, their relations with the classical Ljusternik--Schnirelman theory are given.  In Section~\ref{sec:mainthms} we prove our main results including Theorems~\ref{thm:lls} and \ref{thm:ll}. In Section~\ref{sec:last} we prove Theorems~\ref{thm: Arnol'dC}--\ref{thm:more}. %In Section~\ref{sec:remarks} we  summarize some directions of further study of Ljusternik-Schnirelman theory.

\section*{Acknowledgements}
I am deeply indebted to Weiwei Wu and Jun Zhang for helpful discussions when preparing this paper. Many ideas involved in this work stem from the excellent work of Hofer and Zehnder~\cite{HZ}, Schwarz~\cite{Sc,Sc2} and Ginzburg and G\"{u}rel~\cite{GG}. I benefited a lot from the fascinating papers by Biran and Cornea~\cite{BC,BC2,BC3}, Leclercq and Zapolsky~\cite{LZ} and Kislev and Shelukhin~\cite{KS}. Without their pioneering works, the present paper would be impossible to finish. I thank all of them. I am  grateful to Professor Octav Cornea for valuable remarks and helpful suggestions. I thank Professor Lev Buhovsky for explaining to me the main result in~\cite{BHS1}. I thank Professor Jinxin Xue and my colleague Hong Huang for many useful discussions. I wish to thank Professor Guangcun Lu for helpful remarks and for pointing out to me the work of Givental~\cite{Gi}. Finally, I would like to express my deep gratitude to the anonymous referee for a very careful reading and many valuable suggestions which in practice largely help to improve the presentation of this article, and in particular bring about the appendices. The author is partially supported by NSFC 11701313 and NSFC 12271285, and the Fundamental Research Funds for the Central Universities 2018NTST18 at Beijing Normal University.

%Finally, I would like to mention that close to the completion of the paper I am surprised to find that at the end of the paper~\cite{KS} Kislev and Shelukhin said that they were planning to use Ljusternik--Schnirelman theory to give a lower bound of the number of monotone Lagrangian intersections but so far I do not know whether they have finished their work. warmly thank the referee for a very careful reading

\section{Preliminaries}\label{sec:pre}

\subsection{Floer homology}
In this section we recall the construction of  Lagrangian  and Hamiltonian Floer homology.

\subsubsection{Lagrangian Floer homology}\label{subsec:lfh}
Given a Hamiltonian $H\in\cH$ and a Lagrangian $L\subset M$ we consider the space of  contractible chords relative to $L$
$$\cP_L=\big\{x:[0,1]\to M|x(0), x(1)\in L\; \hbox{and } [x]=0\in\pi_1(M,L)\big\}.$$
For every $x\in\cP_L$, there is a capping $\overbar{x}:\bD\cap\bH\to M$ such that $\overbar{x}|_{\bD\cap\R}=x$ and $\overbar{x}|_{\partial\bD\cap\{\im(z)\geq 0\}}\subset L$, where $\bD=\{z\in\C:|z|\leq 1\}$ and $\bH=\{z\in\C:\im(z)\geq 0\}$. Two cappings $\overbar{x},\overbar{x}'$ are said to be equivalent if the glued map $v=\overbar{x}\sharp\overbar{x}':(\bD,\partial\bD)\to (M,L)$, defined by $\overbar{x}(z)$ for $z\in\bD\cap\bH$
and $\overbar{x}'(\overline{z})$ for $z\in\bD\cap\overbar{\bH}$, satisfies
$\omega(v)=\mu_L(v)=0$. Here $\overbar{\bH}=\{z\in\C:\im(z)\leq 0\}$.  Denote by $\widetilde{\cP}_L$ the cover of $\cP_L$ consisting of all equivalence classes $[x,\overbar{x}]$ of the pairs $(x,\overbar{x})$.
The \emph{action} of $[x,\overbar{x}]$ is given by
$$\ahl([x,\ox])=\int^1_0H(t,x(t))-\int_{\ox}\omega.$$

We denote by  $\spec(H,L)$ the set of critical values of $\ahl$ on $\widetilde{\cP}_L$. It is well known that for a monotone Lagrangian $L$, the set $\spec(H,L)$ is a closed nowhere dense subset of $\R$, see for instance~\cite[Lemma~30]{LZ}. 

%$\cO(L,H)=\cup_{\eta\in\pi_0\cP_L}\cO_\eta(L,H)$
Suppose now that $(H,L)$ is non-degenerate, i.e.,   $\varphi^1_H(L)$ intersects $L$ transversely. 
We denote by $\cO(L,H)$ the set of all Hamiltonian chords of $H$ from $L$ to $L$ representing the zero class in $\pi_1(M,L)$.  It is easy to see that its cover $\widetilde{\cO}(L,H)$ consists of the critical points of $\ahl$ on $\widetilde{\cP}_L$. 

Take a base point $\eta_0=[x_0,\overbar{x}_0]\in \widetilde{\cP}_L$ and define the \emph{index} of each element $\widetilde{x}=[x,\overbar{x}]\in \widetilde{\cO}(L,H)$ by $|\widetilde{x}|=\mu_{V}(\widetilde{x},\eta_0)$ where $\mu_V$ is the Viterbo-Maslov index.  For a detailed construction of this index we refer to~\cite{Vi3}. Here we remark that $\mu_V$ is a relative Maslov index which depends on the choice of a smooth map
$u:[0,1]^2\to M$ satisfying $u(0,t)=x_0(t), u(1,t)=x(t)$ and $u(s,0),u(s,1)\in L$. A different choice of the base point in  $\widetilde{\cP}_L$ gives rise to a shift of degrees. To normalize the index we require that if $H$ is a lift of a Morse function $f$ on $L$ with $\|\cdot\|_{C^2}$-norm small enough to a Weinstein neighborhood of $L$, then for a constant path and the constant capping $\overbar{q}$ at a critical point $q$ of $f$, we have $\mu([q,\overbar{q}])=\ind_f(q)$. 

For each $k\in\Z$, we consider the $\Z_2$-vector space 
 $CF_k(L,H;\Z_2)$  freely generated by elements $\widetilde{x}\in\widetilde{\cO}(L,H)$ with grading $|\widetilde{x}|=k$. 
 Note that for each $A\in\pi_2(M,L)$ and each $[x,\overbar{x}]\in\widetilde{\cO}(L,H)$,  $|[x,\overbar{x}\sharp A]|=|[x,\overbar{x}]|-\mu(A)$.  
 In the monotone case we see that $CF_k(L,H;\Z_2)$ is a finite dimension vector space by degree reasons. We define
 $$CF_*(L,H;\Z_2):=\bigoplus_{k\in\Z}CF_k(L,H;\Z_2)$$ 
 which is a $\Z$-graded vector space over $\Z_2$. 

Fix a family of almost complex structures $\{J_t\}_{t\in[0,1]}$ so that $J_t\in\cJ$ for each $t\in[0,1]$. 
For every $\widetilde{x}=[x,\overbar{x}]\in \widetilde{\cO}(L,H)$, we define the differential
$$d_{H,J}\widetilde{x}=\sum\sharp_2\cM(\widetilde{x},\widetilde{y})\widetilde{y},$$
where $\cM(\widetilde{x},\widetilde{y})$ is the moduli space of solutions $u:\R\times[0,1]\to M$ of the Floer  equation
$$\partial_su+J\partial_tu+\nabla H_t(u)=0$$
with the boundary condition $u(\R\times\{0,1\})\subset L$, and $\sharp_2\cM(\widetilde{x},\widetilde{y})$ denotes the number of its elements modulo $2$. Here the sum is taken over all $\widetilde{y}=[y,\overbar{y}]$ satisfying $\overline{y}=\overline{x}\sharp u$ and  $\mu(\widetilde{x})-\mu(\widetilde{y})=1$. 
By a standard Gromov-Floer compactness argument, one gets $d_{H,J}\circ d_{H,J}=0$.

Denote by $CF_*(L;H)$ the Floer complex consists of elements having the formal sum $$\sum_{[x,\ox]\in \widetilde{\cO}(L,H)} a_{[x,\ox]}[x,\ox],\quad a_{[x,\ox]}\in\Z_2$$ with the property that for any $a\in\R$,
\[\sharp\big\{[x,\ox]\in \widetilde{\cO}(L,H)\big|\;a_{[x,\ox]}\neq 0,\;\ahl([x,\ox])\geq a\big\}<\infty.\]
The Novikov field $\Lambda$ acts on $CF_*(L;H)$ by $$t^{\bar{\mu}(A)}\cdot [x,\overbar{x}]=[x,\overbar{x}\sharp A].$$
On the Novikov ring and the Floer complex as a $\Lambda$-module, we refer to~\cite{HS} for a detailed discussion. It is easy to verify that
\[CF_*(L;H)=CF_*(L,H;\Z_2)\otimes_{\Z_2}\Lambda/\sim\]
where $\sim$ is given by $[x,\overbar{x}]\otimes t^{\bar{\mu}(A)}\sim[x,\overbar{x}\sharp A]$.

Extending $d_{H,J}$ by linearity over $\Lambda$, one obtains a  finite rank free complex $(CF_*(L;H),d_{H,J})$ over $\Lambda$.  The \textit{Lagrangian Floer homology} is defined to be the homology of this complex, and is denoted by $HF_*(L;H,J)$. It can be shown that this homology is  independent of the choice of a family of almost complex structures, and invariant under Hamiltonian perturbations.
In particular, there exists a canonical isomorphism $HF_*(L;H,J)\cong HF_*(L;K,J')$ for any two pairs $(H,J)$ and $(K,J')$ such that the corresponding  homologies are well-defined.

The complex $CF_*(L;H)$ is filtered by the action $\ahl$ as follows:
$$\ahl\big(\sum [x_k,\ox_k]\otimes\lambda_k\big)=\max\limits_k\big\{\ahl([x_k,\ox_k])+A_L\nu(\lambda_k)\big\}.$$

For $a\in\R\setminus\spec(H,L)$, we define
$$CF_*^a(L;H):=\big\{\beta\in CF_*(L;H)\big|\ahl(\beta)<a\big\}.$$
One can show that $CF_*^a(L;H)$ is a subcomplex of  $CF_*(L;H)$. The homology of this subcomplex is denoted by $HF_*^a(L;H,J)$.

\subsubsection{Hamiltonian Floer homology}\label{sec:hfh}
Let $H\in\cH$, and let $\cO(H)=\{\overbar{\ga}=(\ga,\widehat{\ga})\}/\sim$.
Here each $\ga$ is a contractible $1$-periodic orbit of the Hamiltonian flow of $H$, $\widehat{\ga}:\bD\to M$ is a capping of $\ga$ (i.e., $\widehat{\ga}|_{\partial D}=\ga$), and the equivalence relation $\sim$ satisfies: $\overbar{\ga}\sim\overbar{\ga}'$ if $\ga=\ga'$ and $\omega(\widehat{\ga})=\omega(\widehat{\ga}')$. The action functional on $\cO(H)$ is defined by
$$\cA_H(\overbar{\ga})=\int_{S^1}H(t,\ga(t))dt-\int_D\widehat{\ga}^*\omega.$$

For a generic pair $(H,J)$ of a Hamiltonian $H$ and an almost complex structure $J$, the Floer complex $CF_*(H,J;\Gamma)$ is defined to be the set of elements having the formal sum $$\sum_{[\gamma,\widehat{\gamma}]\in \cO(H)} a_{[\gamma,\widehat{\gamma}]}[\gamma,\widehat{\gamma}],\quad a_{[\gamma,\widehat{\gamma}]}\in\Z_2$$ with the property that for any $a\in\R$,
\[\sharp\big\{[\gamma,\widehat{\gamma}]\in \cO(H)\big|\;a_{[\gamma,\widehat{\gamma}]}\neq 0,\;\cA_H([\gamma,\widehat{\gamma}])\geq a\big\}<\infty.\]

%The  complex $CF_*(H,J;\Gamma)$ is filtered by the values of $\cA_H$.

The Novikov ring $\Gamma$ (recall that $\Gamma=\Z_2[s^{-1},s]]$, see~Section~\ref{subsec:notation})  acts on $CF_*(H,J;\Gamma)$
 by $s\cdot(\ga,\widehat{\ga})=(\ga,\widehat{\ga}\sharp A)$ with $\omega(\widehat{\ga}\sharp A)=\omega(\widehat{\ga})+2\kappa_LC_M$. We now extend the functional $\cA_H$ to the complex $CF(H,J;\Lambda)=CF_*(H,J;\Gamma)\tens_\Gamma\Lambda$
by $$\cA_H(\overbar{\ga}\tens t^k)=\cA_H(\overbar{\ga})-kA_L.$$ 

Since $\cA_H(s\cdot\overbar{\ga}\tens t^k)=\cA_H(\overbar{\ga}\tens  t^{2C_M/N_L}\cdot t^k$), the above extension is well-defined. 
The resulting homology is independent of the choice of $J$ by the standard continuation map, and hence we denote it by $HF_*(H,\Lambda)$. Given $\nu\in\R$, we denote by $CF^\nu_*(H,J;\Lambda)$ the subcomplex of the Floer complex generated by all the elements $\overbar{\ga}\tens\lambda$ of actions at most $\nu$, and $HF^\nu_*(H,\Lambda)$ the corresponding homology.

\subsection{ Quantum homology}\label{sec:qh}
The quantum homology $QH(M)=H(M,\Z_2)\tens_{\Z_2}\Gamma$ of a closed monotone manifold $(M,\omega)$ is a module over the ring $\Gamma=\Z_2[s^{-1},s]]$. Using the degree preserving embedding of rings $\Gamma\hookrightarrow\Lambda$ given by $s\to t^{2C_M/N_L}$, we define the obvious extension of the quantum homology:
$$QH(M,\Lambda)=H(M,\Z_2)\tens_{\Z_2}\Lambda=QH(M)\tens_\Gamma\Lambda.$$

Now we extend the valuation map $\nu$ (see~(\ref{e:val})) from $\Lambda$ to $QH(M,\Lambda)$ as
\begin{equation}\label{e:val0}
I_\omega(a)=A_L\max\big\{\nu(\lambda_k)|a_k\neq 0\big\}
\end{equation}
for $a=\sum_{k}a_k\tens\lambda_k\in QH(M,\Lambda)$ with $a_k\in H(M,\Z_2)$ and $\lambda_k\in\Lambda$.

We endow $QH(M,\Lambda)$ with the quantum intersection product
$$*:QH_i(M,\Lambda)\tens QH_j(M,\Lambda)\longrightarrow QH_{i+j-2n}(M,\Lambda),$$
see McDuff and Salamon~ \cite{MS} for the definition. This homology is an associative ring with unit $[M]\in QH_{2n}(M,\Lambda)$. Clearly, the quantum product has degree $-2n$. Furthermore, using a Morse-theoretical approach to quantum homology one can define a ring isomorphism
$${\rm PSS}:QM(M,\Lambda)\longrightarrow HF(H,\Lambda)$$
which is induced by the Piunikin-Salamon-Schwarz homomorphism
~\cite{PSS,Lu2} $$\widetilde{{\rm PSS}}:C(f,g;\Lambda):=\Z_2\brat{\crit (f)}\tens_{\Z_2} \Lambda\to CF(H,J;\Lambda),$$ where the pair $(f,g)$ is Morse-Smale with respect to the Morse function $f:M\to\R$ and the Riemannian metric $g$ on $M$, and the pair $(H,J)$ is generic so that the Floer complex $CF_*(H,J;\Lambda)$ is well-defined.

\subsection{Lagrangian quantum homology}\label{sec:lqh}
Recall that, by~\cite{BC,BC2,BC3}, if $L$ is a closed monotone Lagrangian with $N_L\geq 2$ then the Lagrangian quantum homology $QH(L)$ is well-defined. We briefly recall the construction as follows.  Let $L$ be a Morse function on $L$ and let $\rho$ be a Riemannian metric on $L$ so that the pair $(f,\rho)$ is Morse-Smale. For an $\omega$-compatible almost complex structure $J$  we define $$C(L;f,\rho,J):=\Z_2\brat{\crit (f)}\tens_{\Z_2} \Lambda$$
as the complex generated by critical points of $f$, and  graded by the Morse indices of these critical points and the grading of $\Lambda$.
In the following $W^u(x)$ denotes the unstable submanifold at $x$ of the negative gradient flow of $f$ in $L$. 
For two points $x,y\in\crit(f)$ and a class $A\in\pi_2(M,L)$, we consider the space of sequences $(u_1,\ldots, u_l)$ of possible length $l\geq 1$ satisfying
\begin{itemize}
	\item $u_i:(\bD,\partial \bD)\to (M,L)$ is a non-constant $J$-holomorphic disk, (recall that $\bD$ denotes the closed unit disk in $\C$).
	\item $u_1(-1)\in W^u(x)$. 
	\item For every $i\in\{1,\ldots,l-1\}$, $u_{i+1}(-1)\in W^u(u_i(1))$.
	\item $y\in W^u(u_l(1))$.
	\item $[u_1]+\ldots +[u_l]=A$.
\end{itemize}

Two sequences $(u_1,\ldots,u_l)$ and $(u'_1,\ldots,u'_{l'})$ are said to be equivalent if $l=l'$ and for every $1\leq i\leq l$ there exists $\tau_i\in\aut (\bD)$ such that $\tau_i(-1)=-1, \tau_i(1)=1$ and $u_i=u_i'\circ \tau_i$. Let $\M_{prl}(x,y;A;f,\rho, J)$ be the quotient space with respect to this equivalence relation. Elements of this space are called \textit{pearly trajectories connecting $x$ to $y$}. A typical pearly trajectory is illustrated in Figure~\ref{fig:pear}.

%%%%%%%%%%%%%%%%%%%%%%%%%%%%%%%%%%%%%%%%%%%%%%%%%%%%%%%%%%%%%%%
\begin{figure}[H]
  \centering
  \includegraphics[scale=0.7]{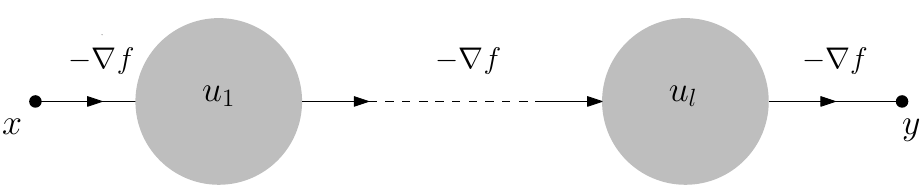}
  \caption{Pearly trajectories connecting $x$ to $y$}\label{fig:pear}
\end{figure}

%%%%%%%%%%%%%%%%%%%%%%%%%%%%%%%%%%%%%%%%%%%%%%%%%%%%%

We extend the definition of the space of pearly trajectories to the case that $A=0$ by defining $\M_{prl}(x,y;0;f,\rho, J)$ to be the space of unparametrized trajectories of the negative gradient flow of $\nabla f$ connecting $x$ to $y$. The virtual dimension of $\M_{prl}(x,y;A;f,\rho, J)$ is given by
$$\virdim\M_{prl}(x,y;A;f,\rho, J)=\ind_f(x)-\ind_f(y)+\mu(A)-1.$$
Here $\ind_f(x)$ denotes the Morse index of the critical point $x$ of $f$.

Suppose that $\virdim\M_{prl}(x,y;A;f,\rho, J)=0$.
For a generic $J\in\cJ$, the moduli space $\M_{prl}(x,y;A;f,\rho, J)$ is a compact $0$-dimensional manifold. Put
$$d(x)=\sum\limits_{y,A}\sharp_2\M_{prl}(x,y;A;f,\rho, J)yt^{\bar{\mu}(A)}$$
and extend $d$ to $C(L;f,\rho,J)$ by linearity over $\Lambda$. The compactness and gluing properties of the moduli spaces of virtual dimension equal to $1$ lead to $d^2=0$. We call the homology of $(C(L;f,\rho,J),d)$ the \textit{Lagrangian quantum homology}. Different choices of datum $\cD=(f,\rho,J)$ give rise to isomorphic resulting homologies by continuation isomorphisms. We denote by $QH(L)$ the abstract Lagrangian quantum homology of $L$ (i.e., the limit of the corresponding direct system of graded modules), and  $QH(L;\cD)$ the homology of  $(C(L;f,\rho,J),d)$ for a specific choice of  datum $\cD=(f,\rho,J)$.

%We note that the homology $QH(L)$ is invariant with respect to the action of the symplectomorphism group. This means that for any symplectomorphism $\varphi$ of $M$ with $L'=\varphi(L)$ there exists a chain map from$C(L;f,\rho,J)$ to $ C(L';f^\varphi,\rho^\varphi,J^\varphi)$ which induces an isomorphism$$\varphi_*:QH(L;f,\rho,J)\longrightarrow QH(L';f^\varphi,\rho^\varphi,J^\varphi),$$ where $f^\varphi=f\circ\varphi^{-1}$, and  $\rho^\varphi,J^\varphi$ are obtained by the pushforward of $\rho, J$ via $\varphi|_L$ and $\varphi$.

We define on the chain complex $(C(L;f,\rho,J),d)$ a map $$\epsilon_L:C(L;f,\rho,J)\longrightarrow\Lambda$$ given by $\epsilon_L(x)=1$ for all $x\in\crit_0(f)$ and $\epsilon_L(x)=0$ for all critical points of $f$ with strict positive index. This is a chain map (cf. page 2917 in~\cite{BC}) since the condition that $N_L\geq 2$ implies that the differential strictly increases the Morse index of each critical point  of index $0$, and since each critical point of index $1$ has two  flow lines emanating from it to critical points of index $0$. The induced map $\epsilon_L$ on $QH(L;\cD)$ is called the \textit{augmentation}.

Now we extend the valuation map $\nu$ (see~(\ref{e:val})) from $\Lambda$ to $C(L;\cD)$ as

$$\nu(x)=\max\big\{\nu(\lambda_k)|\lambda_k\neq 0\big\}$$
for $x=\sum_{k}x_k\otimes\lambda_k\in C(L;\cD)$, where $x_k$ are the critical points of $f$. Define the valuation on $QH(L;\cD)$ by setting
\begin{equation}\label{e:val1}
\nu(\al)=\inf \big\{\nu(x)|[x]=\al\big\}
\end{equation}
and $\nu(0)=-\infty$. A similar valuation was introduced in the work of Entov and Polterovich~\cite{EP}.

\subsubsection{ Lagrangian quantum structures}\label{subsec:lqs}
Following ~\cite{BC,BC2,BC3}, 
we will give a rapid review of two algebraic  structures  of Lagrangian quantum homology which will be useful in this paper.

It is shown that the homology $QH(L)$ carries a supercommutative associative product
$$\circ: QH_i(L)\tens QH_j(L)\longrightarrow QH_{i+j-n}(L),\quad \al\tens \beta\longmapsto \al\circ \beta$$
for every $i,j\in \Z$ with a unit $[L]\in QH_n(L)$.

Also, $QH(L)$ has the structure of a module over the quantum homology  $QH(M,\Lambda)$.  Specifically, for every $i,j\in\Z$, there exists a $\Lambda$-bilinear map
$$\bullet:QH_i(M,\Lambda)\tens QH_j(L)\longrightarrow QH_{i+j-2n}(L),\quad a\tens\al\longmapsto a\bullet\al.$$

These operations endow $QH(L)$ with the structure of a two-sided algebra over the ring $QH(M,\Lambda)$. This means that for any $a\in QH(M,\Lambda)$ and any $\al,\beta\in QH(L)$, we have $$a\bullet(\al\circ \beta)=(a\bullet\al)\circ \beta=\al\circ(a\bullet\beta).$$

%The dual cochain complex is given by $$C^*(L;\cD)=\big(\hom_{\Z_2}\big(\Z_2\brat{\crit (f)},\Z_2\big)\tens \Lambda,d^*\big),$$where for each $x\in\crit_k(f)$ the degree of its dual $x^*\in\hom_{\Z_2}(\Z_2\brat{\crit (f)},\Z_2)$ is $k$, the differential $d^*$ is defined to be the dual of $d$. The cohomology of this complex is called the \textit{Lagrangian quantum cohomology} of $L$ which we denote it by $QH^*(L,\cD)$, and correspondingly, $QH^*(L)$ denotes the abstract Lagrangian quantum cohomology of $L$. Clearly, we have an evaluation $\langle\cdot,\cdot\rangle:QH^*(L)\tens QH_*(L)\to \Lambda$ which is the $\Lambda$-linear extension of the Kronecker pair. Besides, there is a canonical isomorphism $$\cT:QH_k(L)\longrightarrow QH^{n-k}(L)$$ called the \textit{Poincar\'{e} duality map}, which is determined by the bilinear map $$\overbar{\cT}:QH_k(L)\tens QH_l(L)\stackrel{\circ}{\longrightarrow}QH_{k+l-n}(L)\stackrel{\epsilon_L}{\longrightarrow}\Lambda$$ via the relation $\cT(x)(y)=\overbar{\cT}(x\tens y)$.

\subsubsection{Canonical embeddings}\label{subsec:emb}
Although
by the PSS isomorphism (cf. Section~\ref{sec:pss}) we always have $HF(L)\cong QH(L)$,
in general, for a monotone wide Lagrangian $L$, there is no canonical isomorphism $QH_*(L)\cong (H(L,\Z_2)\tens\Lambda)_*$, see~\cite[Section~4.5]{BC}.
However, we notice that for every $p\geq n+1-N_L$, there exists a canonical embedding $i:H_p(L,\Z_2)\tens\Lambda_*\hookrightarrow QH_{p+*}(L)$ for wide Lagrangians $L$, see~\cite[Proposition~4.5.1]{BC}, meaning that for every $p\geq n-N_L+1$, $x$ belonging to the Morse complex $CM_p(L;f,\rho)$ is a $\partial$-cycle if and only if it is a $d$-cycle in $C_p(L;f,\rho,J)$, and $x$ is a $\partial$-boundary if and only if it is a $d$-boundary. In particular, if $N_L\geq n+1$, then this embedding gives an isomorphism between $H(L,\Z_2)\tens\Lambda$ and $QH(L)$. %Throughout this paper, this fact will be used frequently.

\subsection{Piunikhin-Salamon-Schwarz isomorphisms}\label{sec:pss}
For generic $(f,\rho,H,J)$, there are chain morphisms
$\psi: C_*(L;f,\rho,J)\to CF_*(L;H)$ which induce the PSS isomorphisms
\begin{equation}\label{e:PSS}
\Psi_{PSS}:QH_*(L;f,\rho,J)\longrightarrow HF_*(L;H,J).
\end{equation}
Such isomorphisms have been studied in~\cite{Oh2,Oh3,KM,Al,BC,BC2}. Here we mainly follow the construction of Biran and Cornea~\cite{BC,BC2}. Given $q\in\crit(f)$, $\ga=[x,\ox]\in\crit\ahl$ and $A\in\pi_2(M,L)$, consider the configurations of maps $(u_1,\ldots,u_l)$ satisfying
 \begin{itemize}
 	\item each $u_i:(D,\partial D)\to (M,L)$, $i=1,\ldots,l-1$ is a $J$-holomorphic disk (which is allowed to be constant).
 	\item $u_1(-1)\in W^u(q)$.
 	\item For every $i\in\{1,\ldots,l-2\}$, $u_{i+1}(-1)\in W^u(u_i(1))$.
 	\item $u_l:\R\times[0,1]\to M$ is a solution of the equation
 	$$\partial_su+J\partial_tu+\chi(s)\nabla H_t(u)=0,$$
 	and is subject to the conditions: $u_l(\R\times\{0,1\})\subset L$, $u_l(+\infty,t)=x(t)$, $u_l(-\infty,t)\in W^u(u_{l-1}(1))$ for any $t\in[0,1]$,  where $\chi(s)$ is a smooth cutoff function satisfying $\chi(s)=0$ for $s\leq 0$ and $\chi(s)=1$ for $s\geq 1$.
 	\item $[u_1]+\ldots +[u_l\sharp \ox]=A$.
 \end{itemize}
We denote the moduli space of such maps by $\M^\psi(q,\ga):=\M^\psi(q,\ga;J,H,\chi,f,\rho)$. The virtual dimension of this moduli space is
\begin{equation}\label{e:dim}
\virdim\M^\psi(q,\ga)=\ind_f(q)-|\ga|+\mu(A).
\end{equation}

If $\virdim\M^\psi(q,\ga)=0$, then for generic $(f,\rho,H,J)$, one can obtain the appropriate transversality of the standard evaluation map,  so we can define on generators
$$\psi_{PSS}(q)=\sum\limits_{\ga,A}\sharp_2\M^\psi(q,\ga)\ga t^{\bar{\mu}(A)}$$
and extend this map by linearity over $\Lambda$. This extended map is a chain map and induces the PSS isomorphism $\Psi_{PSS}$ in~(\ref{e:PSS}).

\section{Spectral invariants}\label{sec:spectinv}

\subsection{The classical Ljusternik--Schnirelman theory and minmax critical values} Fix a ground field $\F$. Let $X$ be a closed smooth manifold with positive dimension. We will denote by $H_*(X)$ the singular homology of $X$ with coefficient field $\F$. 
Let $f\in C^\infty(X,\R)$ be a smooth function. For any $\nu\in\R$, we define $$X^\nu:=\{x\in X|f(x)<\nu\}.$$

To a non-zero singular homology class $a\in H_*(X)\setminus\{0\}$, we associate a  numerical invariant defined by
$$c_{LS}(a,f)=\inf\{\nu\in\R|a\in\im(i^\nu_*)\},$$
where $i^\nu_*:H_*(X^\nu)\to H_*(X)$ is the map induced by the inclusion $i^\nu:X^\nu\to X$. 
The function $$c_{LS}:H_*(X)\setminus\{0\}\times C^\infty(X,\R)\to\R$$ is called the \textit{minmax critical value selector} of $f$. The following  properties are well-known from classical Ljusternik-Schnirelman theory; see, e.g., \cite{HZ,Cha,Vi2,CLOT,GG}.
%Set $c_{LS}(0,f)=-\infty$.

\begin{prop}\label{pp:minmax}
	The minmax critical value selector $c_{LS}$ satisfies the following properties.
\begin{enumerate}
	%\item[{\rm1.}] {\rm Normalization:} $c_{LS}(a,f)=c$ for any constant function $f\equiv c$.
	\item[{\rm 1.}] $c_{LS}(a,f)$ is a critical value of $f$, and $c_{LS}(ka,f)=c_{LS}(a,f)$ for any nonzero $k\in\F$.
	%\item[{\rm 4.}] {\rm Triangle inequality:} $c_{LS}(a,f+g)\leq c_{LS}(a,f)+c_{LS}(a,g)$.
	\item[{\rm 2.}]  $c_{LS}(a,f)$ is Lipshitz continuous in $f$ with respect to the $C^{0}$-distance.
	\item[{\rm 3.}] Let $[pt]$ and $[X]$ denote the point class and the fundamental class, respectively. Then
$$c_{LS}([pt],f)=\min f\leq c_{LS}(a,f)\leq\max f= c_{LS}([X],f).$$
	%\item[{\rm 6.}] $c_{LS}(a\cap b,f)\leq c_{LS}(a,f)$ for any $b\in H_*(X)$ with $a\cap b\neq 0$.
	\item[{\rm 4.}] If $b\neq k[X]$ and $c_{LS}(a\cap b,f)= c_{LS}(a,f)$, then the set \[\Sigma=\{x\in\crit(f)|f(x)= c_{LS}(a,f)\}\] is homologically non-trivial.
	%\item[{\rm8.}]  $c_{LS}(a+b,f)\leq\max\{c_{LS}(a,f),c_{LS}(b,f)\}$ for any nonzero $a,b\in H_*(X)$.
\end{enumerate}
	
\end{prop}

\subsection{Hamiltonian spectral invariants}\label{subsec:hsi}
In this subsection we recall the definition of Hamiltonian spectral invariants following Oh~\cite{Oh4}, where these spectral invariants are studied for Hamiltonians on closed weakly monotone manifolds. For a further review of this subject, we refer to the papers by Viterbo~\cite{Vi}, Schwarz~\cite{Sc}, Oh~\cite{Oh4,Oh5,Oh6}, and Fukaya, Oh, Ohta and Ono~\cite{FOOO}. Assume that $(M,\omega)$ is a closed monotone symplectic manifold. We take $a\in QH_*(M,\Lambda)=(H(M,\Z_2)\tens\Lambda)_*$ and define the spectral invariant $\sigma(a,H)$ of $a$ as
$$\sigma(a,H)=\inf\big\{\nu\in\R|{\rm PSS}(a)\in\im(i^\nu)\big\},$$
where $i^\nu:HF^\nu(H,\Lambda)\to HF(H,\Lambda)$ is the natural map induced by the inclusion $CF^\nu(H,\Lambda)\hookrightarrow CF(H,\Lambda)$, ${\rm PSS}:QM(M,\Lambda)\longrightarrow HF(H,\Lambda)$ is the Hamiltonian Piunikhin-Salamon-Schwarz isomorphism. 
%By convention $\sigma(0,H)=-\infty$.

%The properties of Hamiltonian spectral invariants can be summarized as follows:
\begin{prop}
	The function $\sigma:QH(M,\Lambda)\setminus\{0\}\times \cH\to\R$ satisfies the properties:
	\begin{enumerate}
		\item[{\rm (HS1)}] {\rm Normalization:} For any $a\in H(M,\Z_2)$ and any $H\in \cH$ with sufficiently small $\|H\|_{C^2}$, we have that  $\sigma(a,H)=c_{LS}(a,H)$. In particular, $\sigma(a,0)=0$.
		%\item[{\rm(HS2)}] {\rm Spectrality:} $\sigma(a,H)\in\spec(H)$.
		%\item[{\rm(HS3)}] {\rm Continuity:} $\sigma(a,H)$ is Lipschitz continuous in $H$ with respect to the $C^0$-distance.
		%\item[{(\rm HS4)}] {\rm Hamiltonian shift:} If $g$ is a function of time, then $\sigma(a,H+g)=\sigma(a,H)+\int^1_0g(t)dt$.
		%\item[{\rm(HS5)}] {\rm Monotonicity:} $\sigma(a,H)\leq\sigma(a,K)$ if $H\leq K$ pointwise.
		%\item[{\rm(HS6)}] {\rm Symplectic invariance:} $\sigma(\varphi_*(a),H)=\sigma(a,\varphi^*H)$ for any symplectomorphism $\varphi$.
		%\item[{\rm(HS7)}] {\rm Triangle inequality:} $\sigma(a*b,H\sharp K)\leq\sigma(a,H)+\sigma(b,K)$.
		\item[{\rm(HS2)}] {\rm Quantum shift:} For $\lambda\in\Lambda$, $\sigma(\lambda a,H)=\sigma(a,H)+I_\omega(\lambda)$.
		\item[{\rm(HS3)}] {\rm Valuation inequality:} $\sigma(a+b,H)\leq\max\{\sigma(a,H),\sigma(b,H)\}$. Moreover, if $\sigma(a,H)\neq\sigma(b,H)$ then this inequality is strict.
		%\item[{\rm(HS10)}] {\rm Homotopy invariance:} $\sigma(a,H)=\sigma(a,K)$, when $\varphi_H=\varphi_K$ in the universal covering of the group of Hamiltonian diffeomorphisms, where $H$ and $K$ are normalized.
	\end{enumerate}
\end{prop}

\subsection{Lagrangian spectral invariants}\label{sec:lsi}
In this subsection we recall the construction of Lagrangian spectral invariants following Leclercq and Zapolsky~\cite{LZ}. These spectral invariants are the Lagrangian counterparts of spectral invariants described in Section~\ref{subsec:hsi}.  In other cases, they were constructed for Lagrangians in the cotangent bundle of a closed manifold~\cite{Oh2,Oh3,Le,MVZ}. Let $H\in\cH$, and let $J\in\cJ$. For generic $(f,\rho,H,J)$, we fix a nonzero homogeneous $\al\in QH_*(L)$ and define the spectral invariant $\ell(\al,H)$ of $\al$ as
$$\ell(\al,H)=\inf\big\{\nu|\Psi_{PSS}(\al)\in \im(i^\nu)\big\}$$
where $i^\nu: HF^\nu(L;H,J)\to HF(L;H,J)$ is the natural map induced by the inclusion of $CF^\nu(L;H)$ into $CF(L;H)$. This invariant is Lipschitz continuous with respect to the Hamiltonian function $H$, i.e.,
for any two nondegenerate $H,K\in \cH$ and $\al\neq 0$, we have
$$\int^1_0\min\limits_M(K_t-H_t)dt\leq \ell(\al,K)-\ell(\al, H)\leq \int^1_0\max\limits_M(K_t-H_t)dt.$$ 

Using this property, one can extend this spectral invariant to a map
$$\ell:QH(L)\times\cH\longrightarrow [-\infty,\infty).$$
%Here by convention we set $\ell(0,H)=-\infty$.

The Lagrangian spectral invariant $\ell$ has the following properties:
\begin{enumerate}
	\item[(LS1)] Normalization: If $c$ is a function of time,  then
	$$\ell(\al,H+c)=\ell(\al,H)+\int^1_0c(t)dt,$$
	and $\ell(\al,0)=A_L\nu(\al)$.
	\item[(LS2)] Spectrality: $\ell(\al,H)\in\spec(H,L)$.
	\item[(LS3)] Quantum shift: $\ell(\al\otimes \lambda,H)=\ell(\al,H)+A_L\nu(\lambda)$ for all $\lambda\in\Lambda$.
	%\item[(LS4)] Symplectic invariance: $\ell(\al,H)=\ell'(\varphi_*(\al),H\circ\varphi^{-1})$ for any symplectomorphism $\varphi$ satisfying $L'=\varphi(L)$, where$\ell':QH(L')\times C^\infty(M\times[0,1],\R)\to\R$ is the corresponding spectral invariant.
	\item[(LS4)] Continuity: $\ell$ is Lipschitz continuous in $H$ with respect to the $C^0$-distance.
	%\item[(LS6)] Monotonicity: $\ell(\al,H)\geq\ell(\al,K)$ for any $\al\in QH_*(L)$ provided that $H\geq K$.
	%\item[(LS7)] Homotopy invariance: $\ell(\al,H)=\ell(\al,K)$, when $\varphi_H=\varphi_K$ in the universal covering of the group of  Hamiltonian diffeomorphisms, and $H,K$ are normalized.
	\item[(LS5)] Triangle inequality: $\ell(\al\circ\beta,H\sharp K)\leq \ell(\al,H)+\ell(\beta,K)$.
	\item[(LS6)] Module structure: Let $K\in\cH$. For all $a\in QH(M,\Lambda)$ and $\al\in QH(L)$, we have
	$\ell(a\bullet\al,H\sharp K)\leq\ell(\al,H)+\sigma(a,K).$
	\item[(LS7)] Valuation inequality: $\ell(\al+\beta,H)\leq\max\{\ell(\al,H),\ell(\beta,H)\}$. Moreover, if $\ell(\al,H)\neq\ell(\beta,H)$ then this inequality is strict.
	\item[(LS8)] Duality: Set $\overbar{H}(t,x)=-H(-t,x)$. Then
	$$\ell(\al,\overbar{H})=-\inf\big\{\ell(\beta,H)|\beta\in QH_{n-k}(L),\;\epsilon_L(\alpha \circ \beta) \ne 0\big\}.$$
   %\item[(LS12)] Lagrangian control: For all $H\in\cH$ we have $$\int^1_0\min\limits_{L}H_tdt\leq \ell(\al,H)-A_L\nu(\al)\leq\int^1_0\max\limits_{L}H_tdt.$$
	 %\item[(LS13)] Non-negativity: $\ell([L],H)+\ell([L],\overbar{H})\geq0$.
\end{enumerate}

\subsection{A relation between $c_{LS}$ and $\ell$}\label{sec:lsi=cls}

\begin{df}
Let $(X,g)$ be a smooth Riemannian manifold, and let $f\in C^\infty_c(X,\R)$. We say that $f$ is \emph{$C^2$-small} if 
\[
\|f\|_{C^2}:=\sup_{x\in M}|f(x)|+\|\nabla f(x)\|+\|\nabla\nabla f(x)\|<\epsilon
\]
for some sufficiently small $\epsilon>0$, where $\nabla$ is the Levi-Civita connection with respect to the metric  $g$, and each $\|\cdot\|$ is the norm induced by $g$. 
\end{df}

Suppose now that $L\subseteq (M,\omega)$ is a closed Lagrangian submanifold equipped with a Riemannian metric $\rho$. For $R>0$, the ball bundle $T^*_RL$ of radius $R$ is defined by
$$T^*_RL:=\{(q,p)\in T^*L|\;\|p\|\leq R\}.$$

 \textbf{In the following, we identify $T^*L$ with some Weinstein neighborhood of $L$ in $M$.} For a $C^2$-small function $f$ on $L$, we pick some $R>0$ so that 
 \[
 L^f:=\{(q,\partial_q f(q))\in T^*L|q\in L\}\subseteq T^*_RL.
 \]

We pick a monotonically decreasing function $\lambda:(0,\infty)\to\R$ such that $\lambda(r)=1$ for $0< r\leq R$ and $\lambda(r)=0$ for $r\geq R+1$.  Consider the following time independent Hamiltonian
$H_f:M\to \R$:
\begin{equation}\notag
H_f(x)=
\begin{cases}
\lambda(\|p\|)f(q)&\hbox{if}\; x=(q,p)\in T^*L;\\
0&\hbox{if}\;  x\in M\setminus T^*L.\\
\end{cases}
\end{equation}
 
 Clearly, $H_f$ is compactly supported in $T^*_{R+1}L$
  and $H_f=f\circ\pi$ on $T^*_RL$, where $\pi:T^*L\to L$ is the projection map. Moreover, we require that the \emph{lift} $H_f$ of $f$ is also $C^2$-small.

In the following 	we will use the canonical embedding $H_p(L,\Z_2)\tens\Lambda_*\hookrightarrow QH_{p+*}(L)$ for a monotone wide Lagrangian $L$ with $p\geq n-N_L+1$, see~Section~\ref{subsec:emb}.

\begin{prop}\label{pp:ls=minmax}
	Let $f:L\to\R$ be a smooth $C^2$-small function, and $H_f\in C^\infty_c(M,\R)$ the corresponding lift. If $L$ is a monotone wide Lagrangian  with $N_L\geq 2$, then
	$$\ell(a,H_f)=c_{LS}(a,f)\quad\hbox{for any}\;a\in H_{*\geq n-N_L+1}(L,\Z_2).$$
\end{prop}

\begin{proof}
	
	Since both $\ell(a,H)$ and $c_{LS}(a,f)$ are continuous in $H$ and $f$ with respect to $C^0$-topology respectively, it suffices to prove the proposition for a  $C^2$-small Morse function $f:L\to\R$ and the corresponding lift $H_f:M\to\R$. In what follows, we  identify the singular homology $H_*(L,\Z_2)$ with the Morse homology $HM(f,\rho)$ of a Morse-Smale pair $(f,\rho)$.

	First, it is easy to see that the Lagrangian $\varphi_{H_f}(L)$ is the graph of $df$ in $T^*L$ and intersects $L$ transversely, and that a critical point $q$ of $f$ is exactly an intersection point between $L$ and $\varphi_{H_f}(L)$. So we get
	\begin{equation}\label{e:action}
	\mathcal{A}_{H_f,L}([q,\overbar{q}])=f(q),
	\end{equation}
	where $\overbar{q}$ is the constant capping of the constant path $q$. Moreover, from our definition (cf. Section~\ref{subsec:lfh}) of the grading of Lagrangian Floer homology we have $\mu(q,\overbar{q})=\ind_f(q)$.

Next, we will prove that for a critical point $q$ of $f$ with $\ind_f(q)\geq n-N_L+1$,  the PSS map $\psi_{PSS}$, given by the moduli space $\M^\psi(q,\ga)=\M^\psi(q,\ga;J,H_f,\chi,f,\rho)$ as in Section~\ref{sec:pss},  sends $q$ to $[q,\overbar{q}]$, where   $\ga$ is a critical point of $\mathcal{A}_{H_f,L}$ having the form $\ga=[q',\overbar{q}'\sharp B]$ for some disk $B$ with boundary  $\partial B\subseteq L$ passing through $q'\in\crit(f)$.

To this end, firstly, we prove that  for $\|f\|_{C^2}$  small enough, the last component $u_l$ of each curve $(u_1,\ldots,u_l)\in \M^\psi(q,\ga)$ 
must be in $T^*_{R+1}L$. 
	
Note that the metric $\rho$ on $L$ defines a natural almost complex structure $J_\rho$ on $T^*L$ which is compatible with the symplectic form $\omega(\cdot,\cdot)=-G_\rho (\cdot,J_\rho \cdot)$, where $G_\rho$ is the induced metric on $T^*L$ by $\rho$. We now extend $J_\rho$ and $
G_\rho$ from the Weinstein neighborhood of $L$ to $M$ so that the extended Riemannian metric and almost complex structure are compactible and $L$ is totally geodesic with respect to the corresponding metric, and denote it by $J$ and $G$, respectively. For the reason why such extension exists, we refer to~\cite[Exercise~4.16]{MS}. 
	
	%We claim that the last component of the PSS curve $(u_1,\ldots,u_l)$ in $\M^\psi(q,\ga)$ entirely lie  in $T^*L$ (identified with a Weinstein neighborhood of $L$). %Arguing by contradiction that the image of some Floer strip $u_l$ does not lie in the Weinstein neighborhood of $L$ (being identified with $T^*L$). 
	
	%Since $u_l(\partial \mathbb{D})\subseteq L$, for any submanifold $N\subseteq M$ containing $L$for any $r>0$, $\im (u_l)\cap \partial T^*_r L\neq \emptyset$.

	%Since $H_f$ is compactly supported in the ball bundle $T^*_{R+1}L$ with contact boundary (a Liouville domain) and $J$ is of contact type near the boundary, by a standard maximum principle (see for instance \cite[Lemma~D.2]{Rit} or \cite{Go1}), 

	When the virtual dimension of $\M^\psi(q,\ga)$ is zero,  by the inequality $\ind_f(q)\geq n-N_L+1$, we deduce from the dimension formula~(\ref{e:dim}) that 
	\[\mu(u_1)+\cdots+ \mu(u_l)=0\]
	and $\ind_f(q)=\ind_f(q')$. We notice that $u_l$ is a continuous disk with boundary on $L$ because its positive asymptotic is a constant chord $q'$, and its negative asymptotic is a point $x$ in $L$. 
	Then from the energy identity 
	$$E(u_l)=\mathcal{A}_{H_f,L}([x,\overline{q'}\sharp (-u_l)])-\mathcal{A}_{H_f,L}([q',\overline{q'}])+\int^{+\infty}_{-\infty}\int^1_{0}\chi'(s)H_f(u_l(s,t))dsdt,$$  and the non-negativity of energy we deduce that 
	$\mu(u_l)\geq 0$. Otherwise, since $L$ is monotone, the last component would have negative energy $E(u_l)=\omega([u_l])+O(\delta)$ where $H_f$ has a $C^2$-norm bounded by $\delta$. 
	Therefore, we have
	\[\mu(u_1)=\cdots=\mu(u_l)=0,\]
and so all the $J$-holomorphic disks $u_1,\ldots,u_{l-1}$ are constants. 

Moreover, $E(u_l)=O(\delta)$. In this case, for sufficiently small $\delta>0$, we have a point-wise upper bound 
\begin{equation}\label{e:energy}
\|\partial_s u_l(s,t)\|\leq O(E(u_l)^{1/4})\quad\forall\;(s,t)\in\R\times[0,1].
\end{equation}
 Our argument follows closely the line of reasoning in~\cite[Appendix~B]{RS} and \cite[Section~1.5]{Sal2}. The inequality (\ref{e:energy}) is a consequence of the fact that the energy density $F=\|\partial_s u\|^2/2$ satisfies the differential inequality
 \begin{equation}\label{e:mean}
	\Delta F\geq -aF^2-b
\end{equation}
 for two constants $a,b\geq 0$, where 
$u$ is any solution of the Floer equation
\begin{equation}\notag
	\begin{cases}
	\partial_su+J\partial_tu+\chi(s)\nabla H_f(u)=0,\\
	u(\R\times\{0,1\})\subset L.\\
	\end{cases}
\end{equation}
For a detailed proof of (\ref{e:mean}), we refer to~\cite[Section~5]{Sal1}. Here we point out that the extra factor $\chi(s)$ appearing in the Floer equation does not effect the ultimate estimate because all derivatives of the function $\chi$ are bounded, and $\|H_f\|_{C^2}$ is small enough.

 Denote 
$$\bH=\{z\in\C:\im(z)\geq 0\},\quad H_r(z_0)=\{z\in\bH: |z-z_0|<r\}\;\hbox{with}\; z_0\in\bH.$$

We will show that 
\begin{equation}\label{e:pointest}
F(0,0)\leq O(E(u_l)^{1/2}).
\end{equation} 
For other points in the strip $\R\times [0,1]$, the proof is similar. Then (\ref{e:energy}) follows from similar estimates for all points in $\R\times [0,1]$. 

 In the proof of the inequality (\ref{e:pointest}), we abbreviate $u=u_l$. Denote $\xi=\partial_s u,\eta=\partial_t u$. 
Consider the normal derivative of $F$ on $H_r(0)\cap\R$ to obtain 
\begin{eqnarray}
\partial_tF(s,0)&=&\langle\xi(s,0),\nabla_t\xi(s,0)\rangle\notag\\
&=&\langle\xi(s,0),\nabla_s\eta(s,0)\rangle\quad\hbox{by using}\;\nabla_s\eta=\nabla_t\xi\notag\\
&=&\big\langle\xi(s,0),\nabla_s\big(J(u(s,0))\xi(s,0)+\chi(s)J(u(s,0))\nabla H_f(u(s,0))\big)\big\rangle\notag\\
&=&\big\langle\xi(s,0),\nabla_\xi J(u(s,0))\xi(s,0)\big\rangle+\big\langle\xi(s,0), J(u(s,0))\nabla_s\xi(s,0)\big\rangle\label{e:perp}\\
&&+\big\langle\xi(s,0),\chi'(s)J(u(s,0))\nabla H_f(u(s,0))+\chi(s)\nabla_s\big(J\nabla H_f\big)(s,0)\big\rangle\notag\\
&=&\chi'(s)\big\langle\xi(s,0),J(u(s,0))\nabla H_f(u(s,0))\big\rangle+\chi(s)\big\langle \xi(s,0),\nabla_s\big(J\nabla H_f\big)(s,0)\big\rangle\notag\\
&=&-\chi(s)\langle \nabla_s\xi(s,0),J(u(s,0))\nabla H_f(u(s,0))\big\rangle\label{e:perp1}\\
&=&\chi(s)\langle J(u(s,0))\nabla_s\xi(s,0),\nabla H_f(u(s,0))\big\rangle\label{e:perp2}\\
&=&0\notag.
\end{eqnarray}
The first term in (\ref{e:perp}) vanishes because $\nabla_\xi J(u)$ is skew-symmetric with respect to the metric $G$~(cf. \cite[Lemma~4.1.14]{MS}). The second term in (\ref{e:perp}) vanishes because $L$ is totally geodesic, and hence we have $\nabla_s\xi(s,0)\in T_{u(s,0)}L$ and $J(u(s,0))\nabla_s\xi(s,0)$ is orthogonal to $T_{u(s,0)}L$. The equality (\ref{e:perp1}) holds due to the following reasons. The metric $
\rho$ determines an isometry $TL\to T^*L$, and a direct summand of the vertical bundle $T^vT^*L$ and the horizontal bundle $T^hT^*L$ with isomorphisms
\[T_xT^*L=T^v_xT^*L\oplus T^h_xT^*L\cong T_qL\oplus T_q^*L\cong T_qL\oplus T_qL,\quad x=(q,p)\in T^*L.\]
With respect to these splittings, 
using $u(s,0)\in L$ and writing $u=(q,p)$, we find that 
\[\nabla H_f(u(s,0))=\big(\hbox{grad}_\rho f(q(s,0)),0\big)\in T_{q(s,0)}L\oplus\{0\}.\]
The almost complex structure $J$ on $T^*L$ has the form
\[
J(q,p)=\begin{pmatrix}
  0 & I \\
  -I & 0
\end{pmatrix},\quad (q,p)\in T^*L,
\] 
and the metric $G$ on $T^*L$ has the form
\[
G(q,p)=\begin{pmatrix}
  \rho & 0 \\
  0 & \rho
\end{pmatrix},\quad (q,p)\in T^*L.
\] 
So $\xi(s,0)=(\partial_sq(s,0),0)$ is orthogonal to $J(u(s,0))\nabla H_f(u(s,0))$, and hence we get (\ref{e:perp1}). The term in (\ref{e:perp2}) vanishes because $J(u(s,0))\nabla_s\xi(s,0)$ is orthogonal to $T_{u(s,0)}L$ and $\nabla H_f(u(s,0))\in T_{u(s,0)}L$. 

Consequently, one can extend $F$ by reflection to a twice continuously differential function on the open disc $B_r(0):=\{z\in\C:|z-z_0|<r\}$. We denote the extended function as $F$. It holds that $F(\overbar{z})=F(z)$ and $\Delta F\geq -a'F^2-b'$ for some constants $a',b'\geq 0$. Since 
$$\int_{B_r(0)}F\leq\frac{1}{2}E(u_l)=O(\delta)$$ for $\delta$ small enough,  it follows from  Lemma~\ref{lem:meanval} that
\[
F(0,0)\leq \frac{b' r^2}{8}+\frac{2}{\pi r^2}\int_{B_r(0)}F.	
\]
Let $E=E(u_l)$ and $r=E^{1/4}$. Then we find
\[
F(0,0)\leq 	\frac{b' E^{1/2}}{8}+\frac{2}{\pi E^{1/2}}\cdot\frac{1}{2}E=\bigg(\frac{b'}{8}+\frac{1}{\pi}\bigg)E^{1/2}.	
\]
So we have (\ref{e:pointest}).

By (\ref{e:energy}), $\|\partial_t u_l\|$ satisfies the point-wise estimate
\[\|\partial_t u_l(s,t)\|=\|\partial_su_l(s,t)+\chi(s)\nabla H_f(u_l(s,t))\|\leq O(\delta^{1/4}).\]

Note that $u_l (\R\times\{0,1\})\subseteq L$, this implies that $u_l$ must be in $T^*_rL$ for any given $r\geq R+1$ provided that $\delta>0$ is small enough, and thus satisfies the following Floer equation:
	   \begin{equation}\label{e:floereq}
	\begin{cases}
	\partial_su+J_\rho\partial_tu+\chi(s)\nabla H_f(u)=0,\\
	u(\R\times\{0,1\})\subset L.\\
	\end{cases}
	\end{equation}

Now we prove that $u_l$ has to be independent of the variable $t$.  
We write $u_l=(q,p)$ and let 
\[ g(s)=\frac{1}{2}\int^1_0|p(s,t)|^2dt.\] 

Suppose that $s_0$ is a maximum point of $g$. 
Then for $\|f\|_{C^2}$ small enough, it follows from Lemma~\ref{lem:apriori}	that 
	\[0\geq\frac{\partial^2 g}{ds^2}(s_0)\geq \frac{1}{2}\int^1_0\big(|\nabla_s p|^2(s_0,t)+|\nabla_t p|^2(s_0,t)\big)dt.\]
The non-negativity of the right hand side of the above inequalities implies that 
\[|\nabla_s p|(s_0,t)\equiv0\equiv|\nabla_t p|(s_0,t)\]
for all $t\in[0,1]$. The condition $u_l(\R\times\{0,1\})\subset L$ implies that $p(s,0)=p(s,1)=0$ for all $s$. Therefore, 
\[\langle p(s_0,t),p(s_0,t)\rangle\equiv 0,\quad \forall t\in[0,1]\] 
and so the maximum of $g$ satisfies $g(s_0)=0$. Then the non-negativity of $g$ implies that $g\equiv 0$, and hence $p(s,t)\equiv 0$ for all $s,t$. Substituting $u_l=(q,0)$ into (\ref{e:fleq}), we obtain that
\[\frac{\partial q}{\partial s}(s)=-\chi(s)\hbox{grad}_\rho f(q(s)),\]
and that $\partial_t q\equiv 0$. Hence, $q$ is independent of the variable $t$, and so does $u_l$.  

 Summing up, if the virtual dimension of $\M^\psi(q,\ga)$ is zero,  then  there are no non-trivial holomorphic disks $u_i$, $i=1,\ldots,l-1$, $u_l$ is a flow line of $-\chi(s)\hbox{grad}_\rho f$, and $q,q'$ have the same Morse index. Moreover, we have $\ga=(q,\overbar{q})$ because there are no non-trivial gradient flow lines between two critical points with the same index for the Morse-Smale pair $(f,\rho)$.  
 So $\M^\psi(q,\ga)$ consists of negative gradient flow lines of $f$ (up to a reparametrization) from $q$ to $\ga=(q,\overbar{q})$. This implies that the only element of $\M^\psi(q,\ga)$ is the constant flow line by dimension reasons, and hence $\psi_{PSS}$ sends $q$ to $[q,\overbar{q}]$.
 
 As a consequence, for $p\geq n-N_L+1$, we have the map
   \begin{equation}\label{e:inclusion}
   \Psi_{PSS}\circ i: H_p(L,\Z_2)\longrightarrow  HF(H_f,J),\quad
   \big[\sum_kq_k\big]\longmapsto \big[\sum_k(q_k,\overbar{q}_k)\big],
   \end{equation}
where $i$ is the embedding map from $H_p(L,\Z_2)\cong H_p(L,\Z_2)\tens\Lambda_0$ to $QH_*(L,f,\rho)$ for each $p\geq n-N_L+1$.

	Combining (\ref{e:action}) and (\ref{e:inclusion}) yields $\ell(a,H_f)=c_{LS}(a,f)$ for all $a\in H_{*\geq n-N_L+1}(L,\Z_2)$.
 \end{proof}
	
\section{Proofs of main results: Lagrangian Ljusternik--Schnirelman theory}\label{sec:mainthms}

\subsection{Proof of Theorem~\ref{thm:lls}}
First of all, by the triangle inequality (LS5) and the normalization property (LS1), we have
$$\ell(\al\circ \beta, H)=\ell(\al\circ \beta, 0\sharp H)\leq \ell(\al,0)+\ell(\beta,H)=\ell(\beta,H)+A_L\nu(\al).$$

\textbf{We prove the second statement.} To prove the strict inequality we proceed in three steps.\\
\noindent\textbf{Step 1.} Take a $C^2$-small Morse function $f\in C^\infty(L,\R)$, and denote $H_f\in C^\infty_c(M,\R)$ its corresponding lift as in Section~\ref{sec:lsi=cls}. Let $\mathcal{S}_f=\{x_1,\ldots, x_k\}$ denote the set of the local maximum points of $f$  (each $x_i$ has Morse index equal to $\dim L$).

We shall prove the following.
\begin{lem}\label{clm:small}
	For any $\al\in \widehat{QH}(L)$, we have
	\begin{equation}\label{e:lsineq}
	\ell(\al,H_f)\leq A_L\nu(\al)+\max\big\{f(x)|x\in \crit(f)\setminus\mathcal{S}_f \big\}.
	\end{equation}
\end{lem}

\noindent {\bf Proof of Lemma~\ref{clm:small}.}
 Since the Lagrangian $\varphi_{H_f}(L)$ is the graph of $df$ in $T^*L$ and intersects $L$ transversely,  a critical point $q$ of $f$ is exactly an intersection point between $L$ and $\varphi_{H_f}(L)$. As before, we have
\begin{equation}\label{e:ptaction}
\mathcal{A}_{H_f,L}([q,\overbar{q}])=f(q),
\end{equation}
where $\overbar{q}$ is the constant capping of the constant path $q$, and $\mu(q,\overbar{q})=\ind_f(q)$.

As in the proof of Proposition~\ref{pp:ls=minmax},  for $q\in\crit(f)$ and $\ga\in\crit(\mathcal{A}_{H_f,L})$, we consider the moduli space $\M^\psi(q,\ga)=\M^\psi(q,\ga;J,H_f,\chi,f,\rho)$. If $\|f\|_{C^2}$ is sufficiently small, then for a generic choice of the pair $(\rho,J)$, one can show as before that the last component $u_l$ of any element $(u_1,\ldots,u_l)$ in
$\M^\psi(q,\ga)$ has to be independent of variable $t$ (here $u_l$ is a flow line of $-\chi(s)\nabla^\rho f$). So $\M^\psi(q,\ga)$, in fact, consists of pearly trajectories connecting $q$ to $q'\in\crit(f)$. If the virtual dimension of $\M^\psi(q,\ga)$ is zero, then by the dimension formula~(\ref{e:dim}) we have $$|\ga|=\ind_f(q)+\mu(A).$$ Now we consider two cases: (1) if $\ga=[q,\overbar{q}]$ then $A=0$, in this case the only pearly trajectory from $q$ to $\ga$  is the constant flow line; (2) if $\ga\neq [q,\overbar{q}]$ then $A\neq 0$ by dimension reasons. Notice that
for each nonzero $\al\in\widehat{QH}(L;f,\rho,J)$, any its representation $\sum_kq_k\tens \lambda_k$ has the property that $\ind_f(q_k)<n$ for each $k$ by degree reasons.  Consequently, we have the map
\begin{eqnarray}
\Psi_{PSS}: QH(L;f,\rho,J)\longrightarrow  HF_*(H_f,J),\notag\\
\big[\sum_kq_k\otimes\lambda_k\big]\longmapsto \bigg[\sum_k[q_k,\overbar{q}_k]\otimes\lambda_k+\sum_l[q_l',\overbar{q}_l'\sharp A_l]\otimes\lambda_l\bigg],\label{e:image}
\end{eqnarray}

Now we look at the actions of the generators appearing in the right hand side of~(\ref{e:image}). By (\ref{e:ptaction}), for every $k,l$,  we have
\begin{equation}\label{e:sumand1}
\mathcal{A}_{H_f,L}([q_k,\overbar{q}_k]\otimes\lambda_k)=f(q_k)+A_L\nu(\lambda_k)
\end{equation}
and $$\mathcal{A}_{H_f,L}\big([q'_l,\overbar{q}'_l\sharp A_l]\otimes\lambda_l\big)=\mathcal{A}_{H_f,L}\big([q'_l,\overbar{q}'_l]\tens t^{\overbar{\mu}(A_l)}\lambda_l\big)=f(q_l')-\overbar{\mu}(A_l)A_L+A_L\nu(\lambda_l).$$
So if  $|f|_{C^0}<A_L/2$, for all $l$ we have
\begin{eqnarray}\label{e:sumand2}
\mathcal{A}_{H_f,L}\big([q'_l,\overbar{q}'_l\sharp A_l]\otimes\lambda_l\big)&=&f(q_l)+A_L\nu(\lambda_l)+f(q_l')-f(q_l)-\overbar{\mu}(A_l)A_L\notag\\
&<&f(q_l)+A_L\nu(\lambda_l)+(1-\overbar{\mu}(A_l))A_L\notag\\
&\leq& f(q_l)+A_L\nu(\lambda_l),
\end{eqnarray}
where each $q_l\in\crit(f)$ coming from the representation $\sum_kq_k\otimes\lambda_k$ is the starting point of some pearly trajectory with end point $q'_l$, and we have used $\overbar{\mu}(A_l)\geq 1$ (due to $N_L\geq 2$) for each $l$ in the third inequality.

Since $\al\in \widehat{QH}(L)$, 
 each $q_k$ appearing in $\sum_kq_k\otimes\lambda_k$ does not belong to $\mathcal{S}_f$. So we have
$$\max_k f(q_k)\leq\max\big\{f(x)|x\in \crit(f)\setminus\mathcal{S}_f \big\}.$$
This, combining with (\ref{e:image})--(\ref{e:sumand2}), implies that for any representation $\sum_kq_k\otimes\lambda_k $ of $\al$, we have
$$\ahl\bigg(\psi_{PSS}\big(\sum_kq_k\otimes\lambda_k\big)\bigg)\leq \max\big\{f(x)|x\in \crit(f)\setminus\mathcal{S}_f \big\}+A_L\max_k\nu(\lambda_k)$$
whenever $\|f\|_{C^2}$ is sufficiently small. Therefore, by the definition of the valuation $\nu$ on $QH(L;\cD)$ (cf. (\ref{e:val1})), we conclude the desired inequality. This completes the proof of Lemma~\ref{clm:small}. \qed

\noindent\textbf{Step 2.} Since the intersections of $L$ and $\varphi_H(L)$ are isolated,  one can pick a small open neighborhood $U$ of $L\cap \varphi_H(L)$ in $L$ such that the number of the components of $U$ is equal to that of the intersection points. Let $f:L\to\R$ be a smooth function such that $f=0$ on $\overbar{U}$, $f<0$ on $L\setminus \overbar{U}$, and all critical points in $\crit(f)\setminus \overbar{U}$ are nondegenerate as illustrated in Figure~\ref{fig:func0}. Let $H_f$ be a lift of $f$ as in Section~\ref{sec:lsi=cls}.
Clearly, $f$ is not a Morse function on $L$. Now we modify $f$ into a Morse function which satisfies the conditions of Step 1.  For any $\eta>0$, pick a smooth Morse function $f_\eta:L\to\R$ such that $\|f_\eta-f\|_{C^0}<\eta$, $f_\eta=f$ outside of a small neighborhood of $\overbar{U}$, and  $f_\eta$  has the same non-degenerate critical points with $f$ except for some local maximum points on $U$ (cf. Figure~\ref{fig:func1}).  Denote by $H_{f_\eta}$ the corresponding lift of $f_\eta$ on $M$.

 Given
$0\neq\al\in \widehat{QH}(L)$, by Step 1,  for sufficiently small $\varepsilon>0$ we have
\begin{eqnarray}\label{e:sine}
\ell(\al,\varepsilon H_{f_\eta})&\leq& A_L\nu(\al)+\varepsilon\max\big\{f_\eta(x)|x\in \crit(f_\eta)\setminus \mathcal{S}_{f_\eta} \big\}\notag\\
&\leq&A_L\nu(\al)+\varepsilon\max\big\{f(x)|x\in \crit(f)\setminus\overbar{U} \big\}.
\end{eqnarray}
Since $\ell(\al,H)$ is continuous with respect to $H$ in $C^0$-topology, by letting $\eta\to 0$ we deduce from (\ref{e:sine}) that
\begin{equation}\label{e:sineq}
\ell(\al,\varepsilon H_f)\leq A_L\nu(\al)+\varepsilon\max\big\{f(x)|x\in \crit(f)\setminus\overbar{U} \big\}<A_L\nu(\al).
\end{equation}

\begin{figure}[]
	\centering
	\includegraphics[scale=0.7]{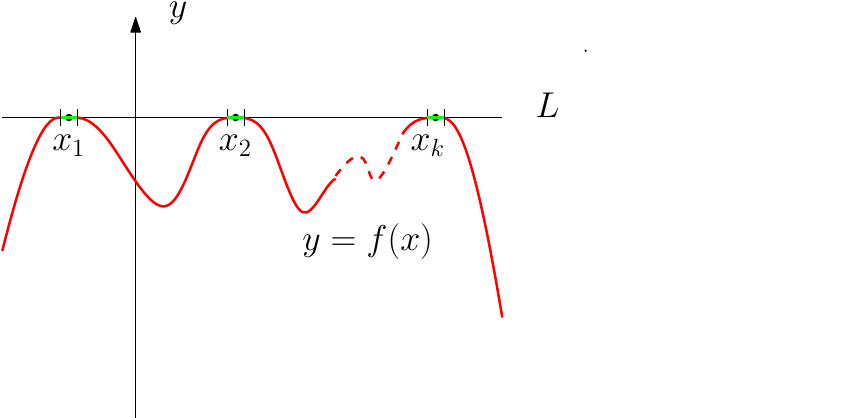}
	\caption{The function $f$.}\label{fig:func0}
	\centering
	\includegraphics[scale=0.7]{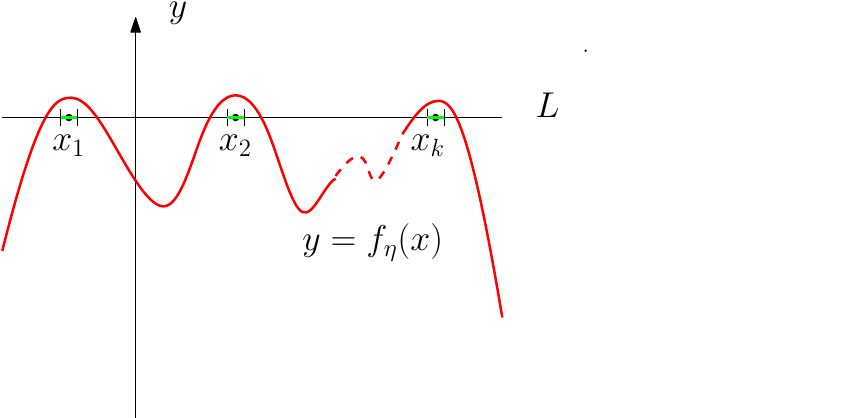}
	\caption{The function $f_\eta$.}\label{fig:func1}
\end{figure}

\noindent\textbf{Step 3.} Let $f$, $U$ and $H_f$ be as in Step 2. We will show the following lemma whose proof given here is inspired by~\cite{BHS1} (in which the symplectic form is exact). Similar arguments in the Hamiltonian case can be found in~\cite{GG,How}.
\begin{lem}\label{clm:sv}
	For sufficiently small $\varepsilon>0$ we have
	$\ell(\al\circ\beta,H)=\ell(\al\circ\beta,\varepsilon H_f\sharp H).$
\end{lem}

\noindent {\bf Proof of Lemma~\ref{clm:sv}.}
In the following, we identify $T^*L$ (resp. the zero section~$o_L$) with a Weinstein neighborhood of $L$ (resp. $L$) in $M$. 

First of all, it is easy to prove that for  sufficiently small $\varepsilon>0$, the Lagrangians $\varphi_{\varepsilon {H_f}}\varphi_H(L)$ and $\varphi_H(L)$ have the same intersection points with $L$ ( due to our specific choice of the function $f$ on $L$ and its lift $H_f$ on $M$).

Next, we will prove that for  sufficiently small $\varepsilon>0$,  the action spectra $\spec(\varepsilon H_f\sharp H,L)$ and $\spec(H,L)$ are the same.

For each $q\in\varphi_H(o_L)\cap o_L$, we have $$\varphi_{\varepsilon H_f}^t \varphi_H^t((\varphi_{\varepsilon H_f} \varphi_H)^{-1}(q))=\varphi_{\varepsilon H_f}^t \varphi_H^t(\varphi_H^{-1}(q)).$$ It is well known that there is a one-to-one correspondence between the set $\varphi_H(L)\cap L$ and the set $\cP_L(H):=\{x\in\cP_L|\dot{x}=X_H(x(t))\}$ of Hamiltonian chords by sending $q\in \varphi_H(L)\cap L$ to $x=\varphi_H^t(\varphi^{-1}_H(q))$. So we obtain a bijection between $\cP_L(H)$ and $\cP_L(\varepsilon H_f\sharp H)$ by mapping $x(t)$ to $\varphi_{\varepsilon H_f}^t(x(t))$.  For simplicity, we write $K=\varepsilon H_f$ and put $u(s,t)=\varphi_K^{st}(x(t))$ for $x(t)\in\cP_L(H)$, where $(s,t)\in S:=[0,1]\times[0,1]$. Then $u(0,t)=x(t)$ and $u(1,t)=\varphi_K^t(x(t))\in\cP_L(K\sharp H)$. To show that $\spec(K\sharp H,L)=\spec(H,L)$, we  consider the map
$$\Theta:\crit(\ahl)\longrightarrow\crit(\cA_{K\sharp H,L}),\quad \big[x(t),\ox\big]\mapsto\big[\varphi_K^t(x(t)), \ox\sharp u\big]$$
as illustrated in Figure~\ref{fig:map}. Clearly, this map is a bijection. 

%%%%%%%%%%%%%%%%%%%%%%%%%%%%%%%%%%%%%%%%%%%%%%%%%%%%%%%%%%%%%%%
\begin{figure}[H]
  \centering
  \includegraphics[scale=0.8]{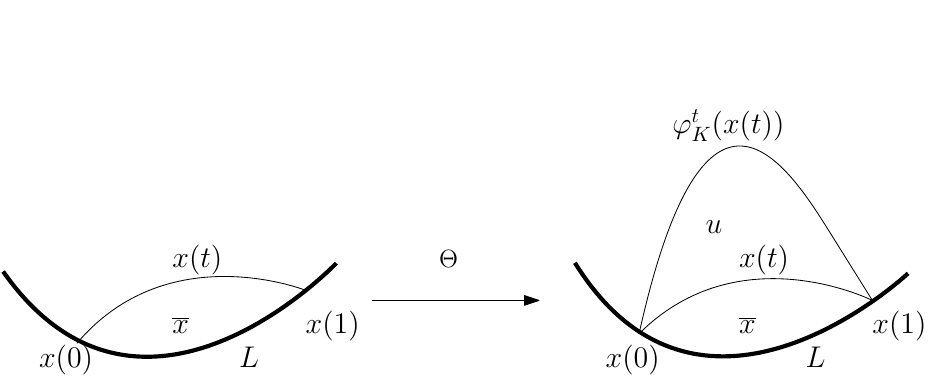}
  \caption{The map $\Theta$}\label{fig:map}
\end{figure}

%%%%%%%%%%%%%%%%%%%%%%%%%%%%%%%%%%%%%%%%%%%%%%%%%%%%%

In the following we will show that $\Theta^*\cA_{K\sharp H,L}=\ahl$. To this end, for every $[x,\ox]\in\crit(\ahl)$, we have
\begin{eqnarray}
\cA_{K\sharp H,L}(\Theta[x,\ox])&=&\int^1_0K(\varphi_K^t(x(t)))dt+\int^1_0H_t\circ(\varphi_K^t)^{-1}(\varphi_K^t(x(t)))dt-\int_Su^*\omega-\int_{\ox}\omega \notag\\
&=&\ahl([x,\ox])+\int^1_0K(x(t))dt-\int^1_0\int^1_0\omega\big(\partial_su,sX_K(\varphi_K^{st}(x(t)))\big)dsdt
\notag\\
&&-\int^1_0\int^1_0\omega\big(\partial_su,d\varphi_K^{st}(\dot{x}(t))\big)dsdt\notag\\
&=&\ahl([x,\ox])+\int^1_0K(x(t))dt-\int^1_0\int^1_0sdK(\varphi_K^{st}(x(t)))[\partial_su]dsdt\notag\\
&&-\int^1_0\int^1_0\omega\big(tX_K(\varphi_K^{st}(x(t))),d\varphi_K^{st}(\dot{x}(t))\big)dsdt\notag\\
&=&\ahl([x,\ox])+\int^1_0K(x(t))dt-\int^1_0\int^1_0
\frac{d}{ds}\big[sK(\varphi_K^{st}(x(t)))\big]dsdt
\notag\\
&&+\int^1_0\int^1_0
\frac{d}{dt}\big[tK(\varphi_K^{st}(x(t)))\big]dsdt\notag\\
&=&\ahl([x,\ox])+\int^1_0K(x(t))dt-\int^1_0K(x(t))dt+K(x(1))\notag\\
&=&\ahl([x,\ox]),\notag
\end{eqnarray}
where in the second and fifth equalities we have used the fact that the value of an autonomous Hamiltonian $H_f$ is invariant along its Hamiltonian flow, and the last equality is implied by $H_f=0$ on the intersections of $\varphi_H(L)$ and $L$. 	
Therefore, the action spectra $\spec(\varepsilon H_f\sharp H,L)$ and $\spec(H,L)$ are the same. 

Now fix $\varepsilon>0$ and consider the family of Lagrangians $\varphi_{s\varepsilon {H_f}}\varphi_H(o_L)$ with $s\in[0,1]$. As above, the action spectra $\spec(s\varepsilon H_f\sharp H,L)$ are all the same. Since the action spectrum is a closed nowhere dense subsets of $\R$ (cf. \cite[Lemma~30]{LZ}), it follows from the spectrality property (LS2) and the continuity property (LS4) that
$\ell(\al\circ\beta, s\varepsilon H_f\sharp H)$ does not depend on $s$. So we have $\ell(\al\circ\beta,H)=\ell(\al\circ\beta,\varepsilon H_f\sharp H)$.
This completes the proof of Lemma~\ref{clm:sv}. \qed

Finally, take $f$, $U$ and $H_f$ to be as in Step 2. For small enough $\varepsilon>0$, by Lemma~\ref{clm:sv}, the triangle inequality~(LS5) and (\ref{e:sineq}), we obtain
\begin{eqnarray}
\ell(\al\circ\beta,H)&=&\ell(\al\circ\beta,\varepsilon H_f\sharp H)\notag\\
&\leq& \ell(\al, \varepsilon H_f)+\ell(\beta, H)\label{eq:ineqlls}\\
&<&A_L\nu(\al)+\ell(\beta, H).\notag
\end{eqnarray}

%Recall that since $L$ is wide, for $N_L\geq n+1$ there exists a canonical isomorphism $(H(L,\Z_2)\tens\Lambda)_*\cong QH_*(L)$, see~Section~\ref{subsec:emb}.
%Let $U\subset L$ be any neighborhood of $L\cap \varphi_H(L)$ in $L$. Let $f\in C^\infty(L,\R)$ be a smooth function such that $f=0$ on $\overbar{U}$ and $f<0$ on $L\setminus \overbar{U}$. Take $H_f\in C^\infty_c(M,\R)$ as a lift of $f$ as in Section~\ref{sec:lsi=cls}.
\textbf{We prove the third statement.} 
Given $\al=\sum_ix_i\tens\lambda_i\in H_{n-N_L<*<n}(L,\Z_2)\otimes \Lambda$,  we get
\begin{eqnarray}
\ell(\al,\varepsilon H_f)&\leq&\sup\limits_i\{\ell(x_i\tens\lambda_i,\varepsilon H_f)\}\notag\\
&\leq&\sup\limits_i\{\ell(x_i,\varepsilon H_f)+A_L\nu(\lambda_i)\}\notag\\
&\leq&\sup\limits_i\{\ell(x_i,\varepsilon H_f)\}+A_L\sup\limits_i\{\nu(\lambda_i)\}\notag\\
&=&\sup\limits_i\{\cls(x_i,\varepsilon f)\}+A_L\nu(\al). \label{eq:ineqls}
\end{eqnarray}
where the first inequality is obtained by (LS7), the second one by (LS3) and the last equality by Proposition~\ref{pp:ls=minmax}.
 
If $\ell(\al\circ\beta,H)=\ell(\beta,H)+A_L\nu(\al)$, then
from (\ref{eq:ineqlls}) and (\ref{eq:ineqls})  we deduce that for sufficiently small $\varepsilon>0$, 
\[A_L\nu(\al)\leq \sup\limits_i\{\cls(x_i,\varepsilon f)\}+A_L\nu(\al).\]
Consequently, we have $0\leq\cls(x_k,\varepsilon f)=\sup_i\{\cls(x_i,\varepsilon f)\}$ for some $k$. By Proposition~\ref{pp:minmax}.3 we have $\cls(x_k,\varepsilon f)\leq\varepsilon\max_Lf=c_{LS}([L],\varepsilon f)=0$. So we have $$\cls(x_k,\varepsilon f)=\cls(x_k\cap [L],\varepsilon f)=\cls([L],\varepsilon f)=0$$ with $x_k\in H_{*<n}(L,\Z_2)$, and hence
by Proposition~\ref{pp:minmax}.4 the zero level set $\overbar{U}$ of $f$ is homologically non-trivial.\qed
\\

The proof of Theorem~\ref{thm:ll} bears a similarity to  the proof of Theorem~\ref{thm:lls}, but is not completely parallel to the former. For the sake of completeness we shall give a sketch of the proof.
\subsection{Proof of Theorem~\ref{thm:ll}}

The non-strict inequality is deduced from the module structure property~(LS6) of $\ell$, the quantum shift property~(HS2) and the normalization property~(HS1) of $\sigma$ as follows: 
$$\ell(a\bullet\al,H)=\ell(a\bullet\al, H\sharp 0)\leq \sigma(a,0)+\ell(\al,H)=I_\omega(a)+\ell(\al,H).$$

\textbf{We prove the second statement.} It suffices to prove that for any two nonzero classes $\alpha\in QH(L)$ and $a=\sum_i x_i \tens_\Gamma\lambda_i$ in $QH(M,\Lambda)$ with each homogeneous class $x_i\in H_{*<2n}(M,\Z_2)$ satisfying 
$$\ell(a\bullet\al,H)=\ell(\al,H)+I_\omega(a),$$
then $L\cap \varphi_H(L)$ is homologically non-trivial in $M$.

First of all, for any smooth $C^2$-small  function $f:M\to\R$, we have that
\begin{equation}\label{e:ineqhs}
\sigma(a,f)\leq \sup_k\{\sigma(z_k,f)\}+I_\omega(a).
\end{equation}
This follows from (HS1), (HS2) and (HS3).

Next, we take a $C^2$-small function $f:M\to\R$ such that $f=0$ on $\overbar{V}$ and $f<0$ on $M\setminus \overbar{V}$, where $V$ is any small neighborhood of $L\cap \varphi_H^{-1}(L)$ in $M$. 
We will show the following lemma whose proof is mainly influenced by the methods used in~\cite{GG,BHS1}. The proof is similar to that of Lemma~\ref{clm:sv}. 
\begin{lem}\label{clm:hsv}
	$\ell(a\bullet\al,H\sharp\varepsilon f)=\ell(a\bullet\al,H)$ whenever $\varepsilon>0$ is small enough.
\end{lem}
\noindent {\bf Proof of Lemma~\ref{clm:hsv}.} We consider a family of spectral invariants $\ell(a\bullet\al,H\sharp s\varepsilon f)$ for $s\in[0,1]$. Note that the function $\ell(a\bullet\al,\cdot):\cH\to\R$ is continuous with respect to the $C^0$-topology and that the action spectrum $\spec(H\sharp s\varepsilon f)$ for each $s\in[0,1]$  is a closed nowhere dense subset of $\R$ (cf. \cite[Lemma~30]{LZ}). It suffices to prove that for $\varepsilon>0$ sufficiently small  the action spectra $\spec(H\sharp s\varepsilon f,L)$ are all the same. To this end, for  $\varepsilon>0$ sufficiently small we will show
\begin{equation}\label{e:intersect}
\varphi_H\circ\varphi_{\varepsilon f}(L)\cap L=\varphi_H(L)\cap L.
\end{equation}
%$\varphi_H(L)\cap L\subseteq \varphi_H\circ\varphi_{\varepsilon f}(L)$ because $\varphi_{\varepsilon f}^t=id$ on $V$ for all $t\in\R$.

On the one hand,  it is easy to see that $\varphi_H\circ\varphi_{\varepsilon f}(L\setminus V)\cap L=\emptyset$ for $\varepsilon>0$ small enough. 
%Otherwise, there exists a sequence of numbers $\varepsilon_i>0$ with $\varepsilon_i\to 0$, and two sequences of points $x_i\in L$, $y_i\in L\setminus V$ such that  $\varphi_{\varepsilon_i f}(y_i)=\varphi_H^{-1}(x_i)$ for all $i$. Since $L$ is compact and $V$ is open in $M$, without loss of generality we may assume that  $x_i\to x\in L$ and $y_i\to y\in L\setminus V$ as $i$ goes to infinity. Then we have $y=\varphi_H^{-1}(x)\subseteq\varphi^{-1}_H(L)\cap L\subseteq V$, which is impossible. 

On the other hand, since  $\varphi_{\varepsilon f}=id$ on $V$,  we obtain $\varphi_H\circ\varphi_{\varepsilon f}(L\cap V)\cap L=\varphi_H(L\cap V)\cap L=\varphi_H(L)\cap L$.
So we have (\ref{e:intersect}).

Next, we will prove that $\spec(H,L)=\spec(H\sharp\varepsilon f,L)$.

Let $\chi:[0,\frac{1}{2}]\to [0,1]$ be a smooth cut-off function so that $\chi=0$ near $t=0$ and $\chi=1$ near $t=\frac{1}{2}$. It is easy to show that the time one flow $\varphi_K$ of the Hamiltonian

\[
K(t,x)=
\begin{cases}\chi'(t)\varepsilon f(x),\quad  \ \ & t\in[0,\frac{1}{2}],\\
\chi'\big(t-\frac{1}{2}\big)H\big(\chi\big(t-\frac{1}{2}\big),x\big), \quad \ \ & t\in[\frac{1}{2},1].
\end{cases}
\]
coincides with the time one flow $\varphi_H\circ\varphi_{\varepsilon f}$ of $H\sharp\varepsilon f$ in the universal covering of the group of  Hamiltonian diffeomorphisms. Moreover, the mean values of $K$ and $H\sharp\varepsilon f$ are the same, i.e., $\int^1_0\int K_t\omega^ndt=\int^1_0\int (H\sharp\varepsilon f)_t\omega^ndt$.
It is well known that $\spec(G,L)=\spec(F,L)$ if $\varphi_G=\varphi_F$ in the universal covering of the group of  Hamiltonian diffeomorphisms with the same mean value, see e.g., \cite{LZ}. Therefore, to prove $\spec(H,L)=\spec(H\sharp\varepsilon f,L)$ we only need to prove that the action spectra
$\spec(H,L)$ and $\spec(K,L)$ are the same. 

For each $q\in\varphi_H(L)\cap L$, due to~(\ref{e:intersect}) the corresponding Hamiltonian chords of $H$ and $K$ are
$\varphi_H^t(\varphi^{-1}_H(q))$ and $\varphi_K^t(\varphi^{-1}_K(q))$ respectively, while, by $\varphi_{\varepsilon f}^t=id$ on $V$ and $\varphi_H^{-1}(q)\in V$, we have
\[
\varphi_K^t(\varphi^{-1}_K(q))=
\begin{cases}\varphi_{\varepsilon f}^{\chi(t)}\circ\varphi_H^{-1}(q)=\varphi_H^{-1}(q),\quad  \ \ & t\in[0,\frac{1}{2}],\\
\varphi^{\chi(t-\frac{1}{2})}_H\circ\varphi^{-1}_H(q), \quad \ \ & t\in[\frac{1}{2},1].
\end{cases}
\]
Consequently, we obtain a bijection
$$\Pi:\crit(\ahl)\longrightarrow\crit(\cA_{K,L}),\quad [x(t),\overbar{x}(s,t)]\mapsto [x'(t),\overbar{x}'(s,t)],$$ where $\overbar{x}(0,t)=q_0$ for some point $q_0\in L$,
$\overbar{x}(1,t)=x(t)$,  $\overbar{x}'(s,t)=\overbar{x}(s,0)$ for $(s,t)\in [0,1]\times[0,\frac{1}{2}]$ and $\overbar{x}'(s,t)=\overbar{x}(s,\chi(t-\frac{1}{2}))$ for
$(s,t)\in [0,1]\times[\frac{1}{2},1]$.
A direct calculation shows that for any $(x,\overbar{x})\in\crit(\ahl)$,
$\cA_{K,L}(\Pi[x,\ox])=\ahl([x,\ox])$, so we have $\spec(H,L)=\spec(K,L)$. 

As above, one can show $\spec(H,L)=\spec(H\sharp s\varepsilon f,L)$ for every $s\in[0,1]$. The spectrality property (LS2) and the continuity property (LS4) imply that
$\ell(a\bullet\al, H\sharp s\varepsilon f)$ does not depend on $s$. So we have 	$\ell(a\bullet\al,H\sharp \varepsilon f)=\ell(a\bullet\al,H)$.   This finishes the proof the lemma.\qed

%Now we are in a position to finish the proof of the strict inequality. Let $f$ be as before, and take $\varepsilon>0$ sufficiently small. It follows from Lemma~\ref{clm:hsv}, the module structure property~(LS9) and (\ref{e:sigma}) that$$\ell(a\bullet\al,H)=\ell(a\bullet\al,H\sharp \varepsilon f)\leq \sigma(a, \varepsilon f)+\ell(\al, H)<\ell(\al, H)+I_\omega(a)$$ as desired.

 Now we are in a position to finish the proof. If $a=\sum_kz_k\otimes_\Gamma\lambda_k\in\widehat{QH}(M,\Lambda)$ with all $ z_k\in H_{*<2n}(M,\Z_2)$ (of pure degree), then by (\ref{e:ineqhs}) for $\varepsilon>0$ small enough, we have
$$\sigma(a,\varepsilon f)\leq \sup_k\{\sigma(z_k,\varepsilon f)\}+I_\omega(a).$$

The moduli structure property~(LS6) and Lemma~\ref{clm:hsv}
imply that
 \[\ell(a\bullet\al,H)=\ell(a\bullet\al,H\sharp \varepsilon f)\leq \sigma(a,\varepsilon f)+\ell(\al,H).\] Then we get $\sigma(z_{k_0},\varepsilon f)\geq \ell(a\bullet\al,H)-\ell(\al,H)-I_\omega(a)=0$ for some $k_0$.  It follows from the normalization property~(HS1) that     $c_{LS}(z_{k_0},\varepsilon f)=\sigma(z_{k_0},\varepsilon f)\geq0$ whenever  $\varepsilon>0$ is sufficiently small. From Proposition~\ref{pp:minmax}.3 we deduce that
$$0\leq c_{LS}(z_{k_0},\varepsilon f)= c_{LS}(z_{k_0}\cap [M],\varepsilon f)\leq c_{LS}( [M],\varepsilon f)\leq 0.$$
Therefore, we obtain $c_{LS}(z_{k_0}\cap [M],\varepsilon f)= c_{LS}( [M],\varepsilon f)$ for some $z_{k_0}\in H_{*<2n}(M,\Z_2)$. Then Proposition~\ref{pp:minmax}.4 implies that the zero level set $\overbar{V}$ of $f$ is homologically non-trivial. since $\varphi_H$ is a diffeomorphism on $M$, as a neighborhood $\overbar{\varphi_H(V)}$  of the intersection $L\cap\varphi_H(L)$ is also homologically non-trivial. The proof is completed.
\qed

\section{Proofs of Theorems~\ref{thm: Arnol'dC}--\ref{thm:more}} \label{sec:last}

\subsection{Proof of Theorem~\ref{thm: Arnol'dC}}
	By definition, there exists a Hamiltonian $H\in\cH$ such that $\varphi=\varphi_H^1$ and $\ga(L,H)<A_L$.
	
	 Without loss of generality we may assume that the intersections of $L$ and $\varphi_H(L)$ are isolated, otherwise, nothing needs to be proved.

%By Theorem~\ref{thm:lls}, for any $\alpha,\beta\in H_{*>n-N_L}(L,\Z_2)$ with $\alpha\neq0$ we have \begin{equation}\label{e:gen}\ell(\al\cap \beta,H)<\ell(\beta,H).\end{equation} Using the valuation inequality property~(LS10) of Lagrangian spectral invariants (see~ Section~\ref{sec:lsi}), one can see that for $\alpha$ having the form $\alpha=\sum \alpha_{i_1}\cap\cdots\cap\alpha_{i_k}$, we still have~(\ref{e:gen}).

By the definition of cup-length (cf. Section~\ref{sec:1}), there exist $u_i\in H_{*<n}(L,\Z_2)$, $i=1\ldots,k$ such that $u_1\cap\cdots\cap u_k=[pt]$ and $cl(L)=k+1$.  Since $H(L,\Z_2)$ is generated as a ring by elements in $H_{\geq n+1-N_L}(L,\Z_2)$, we may further assume that each $u_i\in H_{n-N_L<*<n}(L,\Z_2)$. 

%Otherwise, there exists at least some $u_i\in H_{*\leq n-N_L}(L,\Z_2)$ which has the form$$u_i=\sum \beta_{i_1}\cap\cdots\cap\beta_{i_l}=\sum\beta_{i_1,\ldots,i_l}$$with each $\beta_{i_j}\in H_{n-N_L<*<n}(L,\Z_2)$, $1\leq j\leq l$, where $\beta_{i_1,\ldots,i_l}=\beta_{i_1}\cap\cdots\cap\beta_{i_l}$. Due to $u_1\cap\cdots\cap u_k=[pt]\neq0$, there exists some $\beta_{i_1,\ldots,i_l}\neq0$ such that\[u_1\cap\cdots\cap\beta_{i_1,\ldots,i_l}\cap\cdots \cap u_k\neq 0\]which contradicts to the definition of cup-length whenever $l\geq 2$.

Since for a wide Lagrangian $L$, we have a canonical embedding $H_p(L,\Z_2)\tens\Lambda_*\hookrightarrow QH_{p+*}(L)$ for any $p> n-N_L$ (see Section~\ref{subsec:emb}), 
each $u_i$ can be considered also as a Lagrangian  quantum homology class. 

Consider the sequence $\alpha_0,\dots,\alpha_k$ in $QH(L)$ so that 
\[
\alpha_0=[L]\quad\hbox{and}\quad \alpha_{i}=u_{k-i+1}\circ \alpha_{i-1},\;i=1\ldots,k.
\]
	%$$[L]=\al_0,\al_1,\ldots,\al_k\in QH(L), \quad\hbox{where}\; \al_i=u_{k-i+1}\circ\al_{i-1}.$$

	Since the Lagrangian quantum product $\circ$ is a quantum deformation of the homological intersection product in $H(L,\Z_2)$ (see~\cite{BC,BC2,BC3}), each  $\al_i$ is nonzero.
	In particular, we have
	$$\al_k=[pt]+\sum_{r\geq 1}a_{rN_L}t^r,\quad a_{rN_L}\in H_{rN_L}(L,\Z_2).$$
	 	 This implies that $\epsilon_L([L]\circ \alpha_k)\ne 0$.
	From the Poincar\'{e} duality property~(LS8) of $\ell$
	 we infer that $-\ell([L],\overbar{H})\leq \ell(\al_k,H)$, so we have
	$$\ga(L,H)=\ell([L],H)+\ell([L],\overbar{H})\geq \ell([L],H)-\ell(\al_k,H).$$
	 It follows from  Theorem~\ref{thm:lls} and the spectrality property~(LS2) that there exist $k+1$ elements $\widetilde{x}_i=[x_i,\ox_i]\in\crit(\ahl)$  such that
	 $$\ell(\al_k,H)=\ahl(\widetilde{x}_k)<\ahl(\widetilde{x}_{k-1})<\cdots<\ahl(\widetilde{x}_0)=\ell(\al_0,H).$$
	Clearly, from $\ga(L,H)<A_L$ we exclude the possibility that one of $\{\widetilde{x}_i\}_{i=0}^k$ is the recapping of another.
	Therefore, all $x_i$ are different; equivalently, all elements of $L\cap\varphi(L)$ are distinct. This completes the proof. \qed

\subsection{Proof of Theorem~\ref{thm:two}} Assume that $\varphi\in\cH am(M,\omega)$ is generated by a Hamiltonian $H\in \cH$, i.e., $\varphi=\varphi_H^1$. Since $L$ is non-narrow, $[L]\in QH(L)$ is non-zero. Note that $[M]\bullet [L]=[L]$.
Then we have
\begin{eqnarray}\label{e:triangle}
\ell([L],H)&=&\ell\big([(t^{-\tau}  u_1)*\cdots *u_k]\bullet [L],H\big)\\\notag
&=&\ell\big((t^{-\tau} u_1)\bullet u_2\bullet\cdots\bullet(u_k\bullet [L]),H\big)\\\notag
&=&\ell\big( u_1\bullet u_2\bullet\cdots\bullet(u_k\bullet [L]),H\big)+\tau A_L
\\\notag
&<&\ell\big(u_2\bullet\cdots\bullet(u_k\bullet [L]),H\big)+\tau A_L\\\notag
&<&\cdots
\\\notag
&<&\ell\big(u_k\bullet [L],H\big)+\tau A_L,\\\notag
\end{eqnarray}
where we have used the module structure of Lagrangian quantum homology in the second equality, the quantum shift property~(LS3) in the third equality and Theorem~\ref{thm:ll} in the remaining inequalities. Using Theorem~\ref{thm:ll} again, we obtain $\ell\big(u_k\bullet [L],H\big)< \ell([L],H)+I_\omega(u_k)=\ell([L],H)$.  This together with (\ref{e:triangle}) implies that
\[
\begin{split}
\ell([L],H)-\tau A_L=\ell\big(u_1\bullet\cdots\bullet(u_k\bullet [L]),H\big)<\ell\big( u_2\bullet\cdots\bullet(u_k\bullet [L]),H\big)<\cdots\\
<\ell(u_k\bullet [L],H)<\ell([L],H).
\end{split}
\]
It follows from the spectrality property (LS2) of $\ell$ that
there exist $k+1$ critical points $\widetilde{x}_i=[x_i,\ox_i]$ of $\ahl$   such that
\begin{equation}\label{e:seq}
\ell([L],H)-\tau A_L=\ahl(\widetilde{x}_1)<\cdots<\ahl(\widetilde{x}_{k+1})=\ell([L],H).
\end{equation}

Let $\mathcal{B}$ denote the set consisting of  $a_i:=\ahl(\widetilde{x}_i)$, $i=1,\ldots, k+1$. We define an equivalence relation $\sim$ on  $\mathcal{B}$ by $a_i\sim a_j$ if $(a_i-a_j)/A_L\in\Z$. Denote by $\widehat{\mathcal{B}}$ the set of equivalence classes.
Clearly, $\#\widehat{\mathcal{B}}$ bounds from below the number of Hamiltonian paths $x_i$ (corresponding to the elements in $L\cap\varphi(L)$). So to finish the proof we only need to estimate the former. Note that $[a_1]=[a_{k+1}]$, by (\ref{e:seq}), this class contains at most $\tau+1$ different representative elements in $\mathcal{B}$ , and each other class contains at most $\tau$ different ones. Therefore, we have
$$\sharp \widehat{\mathcal{B}}\geq 1+\bigg\lceil \frac{k+1-\tau-1}{\tau}\bigg\rceil=\bigg\lceil \frac{k}{\tau}\bigg\rceil$$
which concludes the desired result.\qed

\subsection{Proof of Theorem~\ref{thm:more}} The proof is parallel to that of Theorem~\ref{thm:two}. The key idea is to use Theorem~\ref{thm:lls}. We omit the proof.\qed

\section{Hamiltonian periodic orbits and Concluding remarks}\label{sec:remarks}

In Floer theory the Lagrangian (resp. Hamiltonian) spectral invariants serve as the minmax critical value selectors, the Lagrangian (resp. Hamiltonian) Ljusternik--Schnirelman theory seems to be a very useful tool towards the Arnold conjecture for degenerate Lagrangian intersections (resp. degenerate symplectic fixed points), and could be of independent interests. In the following concluding remarks we summarize some directions of further study of this class of objects.

	 Just as we have seen in Theorem~\ref{thm:two} and Theorem~\ref{thm:more}, the lower bounds given by the algebraic  structures of Lagrangian quantum homology (more specifically, by FQF and LFQF) does not depend on the Hamiltonian diffeomorphisms on an ambient symplectic manifold. So making a lower bound estimate of Lagrangian intersections could be translated into a pure algebraic calculation as long as the product and module structures of Lagrangian quantum homology (identically, Lagrangian Floer homology) are known. 
	%\item In Theorem~\ref{thm:lls} for the strict inequality we have made the assumption that the minimal Maslov index  $N_L$ is at least $\dim L+1$. It is enough for our applications given  in the present paper, although this strict inequality still holds for $N_L\geq 2$. In our  forthcoming paper~\cite{Go} we improve Lagrangian Ljusternik--Schnirelman inequality~I to the case that $L$ is wide with $N_L\geq 2$, and furthermore we use it to estimate the number of intersections of Clifford/Chekanov torus in $\C P^n$ with its image of a Hamiltonian diffeomorphism. In particular, for the Cliffold torus $L=\T_{clif}^2\subset\C P^2$ we state: (i) $\sharp(L\cap\varphi(L))\geq 2$ for all  Hamiltonian diffeomorphisms $\varphi$ of $\C P^2$;  (ii) $\sharp(L\cap\varphi(L))\geq 3$ if $\ga(L,\varphi(L))<A_L$.
	 Since Lagrangian spectral invariants generalize  Hamiltonian spectral invariants (see~\cite{LZ}), the Lagrangian Ljusternik--Schnirelman inequality~I in Theorem~\ref{thm:lls} can be viewed naturally as a  partial generalisation of the Hamiltonian Ljusternik--Schnirelman inequality established by Ginzburg and G\"{u}rel~\cite{GG}. Indeed, let $\Delta$ be the Lagrangian diagonal in the product symplectic manifold $(M\times M,\omega\oplus(-\omega))$, then $N_\Delta=2C_M\geq 2$.
    Clearly, if $(M,\omega)$ is monotone then the Lagrangian submanifold $\Delta$ is also monotone.  	
	 Furthermore, there is canonical algebra isomorphism $QH(\Delta)=QH(M)$ so that for every class $\al\in QH(\Delta)$ and every Hamiltonian $H\in C_c^\infty(S^1\times M)$ we have
	$$\sigma(\al,H)=\ell(\al,H\oplus 0),$$
	see~\cite[Theorem~5]{LZ} for a proof. Correspondingly, if $M^{2n}$ be  a  monotone tame symplectic manifold, then for $\al\in \widehat{QH}(M)=H_{2n-2C_M<*<2n}(M,\Z_2)\tens \Gamma$ and $\beta\in QH(M)$ we have the strict inequality %(see~\cite[Proposition~6.2]{GG}) 
		\[
	\sigma(\al*\beta, H)< \sigma(\beta, H)+I_\omega(\al).
	\]
	Indeed, this strict inequality holds for the weaker conditions that  $M$ is weakly monotone (see McDuff and Salamon~\cite{MS} for the definition) with a minor change of the defintion of the Novikov ring $\Gamma$, and that $\al\in \widehat{QH}(M)=H_{<2n}(M,\Z_2)\tens \Gamma$ is arbitrary, see~\cite{GG}.

 In view of this relation between these two Ljusternik--Schnirelman theories, corresponding to Thoerem~\ref{thm: Arnol'dC} one may obtain the following  Chekanov-type results.

We denote by $A_M$ the minimal positive generator of $\omega(\pi_2(M))$, and $\gamma$ the spectral norm given by the spectral invariant $\sigma$:
$$\gamma(\varphi)=\inf_{\varphi=\varphi_H}\sigma([M],H)+\sigma([M],\overbar{H}),\quad\forall \varphi\in\ham(M,\omega).$$ 

 \begin{thm}
	Let $(M,\omega)$ be a monotone symplectic manifold. Then for any $\varphi\in\ham(M,\omega)$,  
	$$\sharp {{\rm Fix}(\varphi)}\geq cl(M)$$
	provided that $\gamma(\varphi)<A_M$.
 \end{thm}

 Similar to Definition~\ref{def:fqf}, one can propose the following definition:
	\begin{df}\label{def:qf}
		Let $M^{2n}$ be  a  monotone tame symplectic manifold $(M^{2n},\omega)$. We say that $M$ has a \textit{ fundamental quantum factorization} (denoted by FQF for short)  \textit{of length $l$ with order $\kappa$} if there exist $u_1,\ldots,u_l\in H_{*<2n}(M,\Z_2)$ such that
		$$s^\kappa[M]= u_1*u_2*\cdots *u_l\quad \hbox{in}\; QH(M).$$
	\end{df}
	Similar to Theorem~\ref{thm:more}, one can obtain
	\begin{thm}\label{thm:fixpoints}
		Let $M^{2n}$ be  a  monotone tame symplectic manifold $(M^{2n},\omega)$.  Suppose that $M$ has a FQF of length $l$ with order $\kappa$.  Let $\varphi$ be any compactly supported Hamiltonian diffeomorphism on $M$ with isolated fixed points. Then $\varphi$ has at least $\lceil l/\kappa\rceil$ fixed points. 
	\end{thm}
	With this theorem at hand, one could easily deduce from the quantum product of quantum homology/cohomology of an explicit symplectic manifold that the least number of fixed points which a Hamiltonian diffeomorphism must have. For instance, for $
	\C P^n$ we have
	\[ 
	h^{*k}= 
	\begin{cases}h^{\cap k},\quad  \ \ & 0\leq k\leq n,\\ 
	[\C P^n]s, \quad \ \ & k=n+1.
	\end{cases} 
	\] 
	Here $h=[\C P^{n-1}]\in H_{2n-2}(\C P^n,\Z_2)$ the class of a hyperplane in the quantum homology $QH(\C P^n)$. So $\C P^n$ has a FQF of length $n+1$ with order $1$, and hence every $\varphi\in\ham(\C P^n)$ must have at least $n+1$ fixed points. This recovers the classical result by Fortune~\cite{Fo}, see also~\cite{FW,Fl3}. Many other related results could also be obtained similarly. A completed list of such examples is not possible to present here, for more examples we refer to~\cite{GG}.
	
 In principle, one could extend Ljusternik--Schnirelman theory to a more general background of spectral invariants, for instance, the \textit{boundary depth} introduced by Usher~\cite{Us,Us2}, 
	\textit{spectral length} introduced by Atallah and Shelukhin~\cite{AS},  \textit{spectral invariants with bulk}~introduced by Fukaya \textit{et al }~\cite{FOOO}, \textit{PHF spectral invariants} introduced by Edtmair and Hutchings~\cite{EH}  
	and so on. After that one may use various versions of Ljusternik--Schnirelman theory to study properties of Hamiltonian dynamics or symplectic geometry,{\footnote{which we will investigate in our future work~\cite{Go4}}} e.g., making more refined estimates of the numbers of Lagrangian intersections,  periodic orbits, Reeb chords, etc. For more applications of Ljusternik--Schnirelman theory we refer to~\cite{Go1,Go2,Go3,Go4}.

%\end{itemize}

\appendix
\section{Mean value inequality for subharmonic functions.}

Denote by
\[
\Delta=\frac{\partial^2}{\partial x_1^2}+\cdots+\frac{\partial^2}{\partial x_n^2}
\]
the Laplacian operator on $\R^n$. Let
$B_r$ denote the round disk of radius $r$ centered at $0$ in $\R^n$. 

The following lemma is a generalization of the mean value inequality for subharmonic functions. It was proved by Robbin and Salamon in~\cite[Appendix A]{RS}. 
\begin{lem}[{\cite[Lemma~A.1]{RS}}]\label{lem:meanval}
Given $\rho>1$, there exists a constant $\hbar=\hbar(\rho,n)$ such that if $u:B_r\to\R$ is a bounded non-negative $C^2$-function that satisfies the inequalities
\[
\Delta u\geq-\lambda- \mu u^{(n+2)/n},\quad\int_{B_r}u<\frac{\hbar}{\mu^{n/2}}
\]
for some constants $\lambda,\mu\geq 0$, then 
\[
u(0)\leq \frac{\lambda r^2}{2n+4}+\frac{\rho}{\rm{Vol}(B_r)}\int_{B_r}u.	
\]
\end{lem}

Here we remark that if $\mu=0$ then the condition $\int_{B_r}u<\rho/\mu^{n/2}$ can be omitted and the estimate for $u(0)$ holds with $\rho=1$.

\section{Apriori estimates.}

Throughout this section, let $\rho$ be a Riemannian metric on $L$, and $J_\rho$ the natural almost complex structure induced by $\rho$ which is compatible with the canonical symplectic form $\sum dq_i\wedge dp_i$. Denote by $\nabla$ the Levi-Civita connection with respect to the metric  $\rho$. Denote by $|\cdot|$ the norm induced by $\rho$. 
Let $\chi(s)$ be a smooth cutoff function satisfying $\chi(s)=0$ for $s\leq 0$ and $\chi(s)=1$ for $s\geq 1$. For $f\in C^\infty(L,\R)$, let $H_f$ denote the lift of $f$ to $T^*L$ given as in Section~\ref{sec:lsi=cls}. The following lemma is a generalization of well-known apriori estimates for Floer strips in $T^*L$. The proof below uses a classical trick by A. Floer~\cite{Fl}. Our argument follows closely the line in~\cite{KO} where Kasturirangan and Oh obtained a similar inequality. 

\begin{lem}\label{lem:apriori}
Let $u:\R\times [0,1]\to T^*L$  be a smooth solution with $u(s,t)=(q(s,t),p(s,t))$ to the equation~(\ref{e:floereq}). 
If $f$ has a sufficiently small $C^2$-norm, then the function \[ g(s)=\frac{1}{2}\int^1_0|p(s,t)|^2dt\] satisfies the differential inequality
\[
\frac{\partial^2 g}{ds^2}\geq \frac{1}{2}
\int^1_0\big(|\nabla_s p|^2+|\nabla_t p|^2\big)dt.
\]
\end{lem}

\begin{proof}
	Following the idea of Floer~\cite[Theorem~2]{Fl}, we rewrite the Floer equation as 
	\begin{equation}\label{e:fleq}
	\begin{cases}
	\frac{\partial q}{\partial s}-\nabla_tp+\chi(s)\lambda(|p|)\hbox{grad}_\rho f(q)=0,\\
	\frac{\partial q}{\partial t}+\nabla_sp+\chi(s)f(q)\lambda'(|p|)\frac{p}{|p|}=0\\
	\end{cases}
	\end{equation}
	  with the boundary condition $p(s,0)=p(s,1)=0$.  Consider
	  \begin{equation}\label{e:2nd}
		\frac{\partial^2 g}{ds^2}=\int^1_0\big(\langle \nabla_s p,\nabla_s p\rangle+\langle \nabla_s\nabla_s p, p\rangle\big)dt.
		\end{equation}
	  To estimate the right hand side of the above equality,  
	  using the second equality in~(\ref{e:fleq}), we obtain
\begin{eqnarray}
	  \nabla_s\nabla_s p&=&\nabla_s\bigg(-\frac{\partial q}{\partial t}-\chi(s)f(q)\lambda'(|p|)\frac{p}{|p|}\bigg)\notag\\
	  &=&-\nabla_t \frac{\partial q}{\partial s}-\chi'(s)f(q)\lambda'(|p|)\frac{p}{|p|}-\chi(s)\nabla_s\bigg(f(q)\lambda'(|p|)\frac{p}{|p|}\bigg)\notag
\end{eqnarray}

By the first equality in~(\ref{e:fleq}), 
\[
\nabla_t \frac{\partial q}{\partial s}=\nabla_t \nabla_t p-\chi(s)\nabla_t \big(\lambda(|p|)\hbox{grad}_\rho f(q)\big),
\]
and so 
\[\nabla_s\nabla_s p=-\nabla_t \nabla_t p+\chi(s)\nabla_t \big(\lambda(|p|)\hbox{grad}_\rho f(q)\big)-\chi'(s)f(q)\lambda'(|p|)\frac{p}{|p|}-\chi(s)\nabla_s\bigg(f(q)\lambda'(|p|)\frac{p}{|p|}\bigg)\]
Hence, 
\begin{eqnarray}
	\int^1_0\langle \nabla_s\nabla_s p, p\rangle dt&=&-\int^1_0 \langle \nabla_t \nabla_t p, p\rangle dt-\int^1_0\chi'(s)f(q)\lambda'(|p|)|p|dt\notag\\
	&&+\chi(s)\int^1_0 \big\langle \nabla_t \big(\lambda(|p|)\hbox{grad}_\rho f(q)\big)-\nabla_s\bigg(f(q)\lambda'(|p|)\frac{p}{|p|}\bigg), p\big\rangle dt\notag\\
	&=&\int^1_0\langle \nabla_t p, \nabla_t p\rangle dt-\chi'(s)\int^1_0f(q)\lambda'(|p|)|p|dt\label{e:parint}\\ &&+\chi(s)\int^1_0 \big\langle \nabla_t \big(\lambda(|p|)\hbox{grad}_\rho f(q)\big)-\nabla_s\bigg(f(q)\lambda'(|p|)\frac{p}{|p|}\bigg), p\big\rangle dt\label{e:derp}
\end{eqnarray}
where in (\ref{e:parint}) we have used integration by parts and the condition $p(s,0)=p(s,1)=0$. Now we consider the two terms in (\ref{e:derp}) to obtain
\[\nabla_t \lambda(|p|)=\lambda'(|p|)\frac{\langle \nabla_t p,p\rangle}{|p|},\]
\[\nabla_t \big(\hbox{grad}_\rho f(q)\big)=\nabla (\hbox{grad}_\rho f)\cdot \frac{\partial q}{\partial t}=\nabla (\hbox{grad}_\rho f)\cdot \bigg(-\nabla_sp-\chi(s)f(q)\frac{\lambda'(|p|)}{|p|}\cdot p\bigg),\]
\[\frac{\partial}{\partial s} \bigg(\frac{\lambda'(|p|)}{|p|}\bigg)=\bigg(\frac{\lambda'(r)}{r}\bigg)'\bigg|_{r=|p|}\cdot \frac{\langle \nabla_t p,p\rangle}{|p|},\]
\[\frac{\partial}{\partial s} f(q)=df(q)\bigg(\frac{\partial q}{\partial s}\bigg)=\big\langle\hbox{grad}_\rho f(q),\nabla_tp-\chi(s)\lambda(|p|)\hbox{grad}_\rho f(q)\big\rangle.\]

By our choices of the functions $\chi$ and $\lambda$, there exists a universal constant $C_1>0$ such that 
\[
\chi(s),\;|\chi'(s)|,\; \lambda(r),\; |\lambda'(r)|,\;\bigg|\frac{\lambda'(r)}{r}\bigg|,\;\bigg|r\bigg(\frac{\lambda'(r)}{r}\bigg)'\bigg|<C_1 
\]
for all $s,r$. Using these inequatilies we obtain 
\begin{eqnarray}
	\bigg|-\chi'(s)\int^1_0f(q)\lambda'(|p|)|p|dt&+&\chi(s)\int^1_0 \bigg\langle \nabla_t \big(\lambda(|p|)\hbox{grad}_\rho f(q)\big)-\nabla_s\bigg(f(q)\lambda'(|p|)\frac{p}{|p|}\bigg), p\bigg\rangle dt\bigg|\notag\\
	&\leq& C_2 \|f\|_{C^2}\int^1_0\big(|\nabla_t p||p|+|\nabla_s p||p|+|p|^2\big)dt\notag\\
	&\leq& C_3\|f\|_{C^2}\int^1_0\big(|\nabla_s p|^2+|\nabla_t p|^2+|p|^2\big)dt\notag\\
	&\leq& C_4\|f\|_{C^2}\int^1_0\big(|\nabla_s p|^2+|\nabla_t p|^2\big)dt\label{e:normest}
\end{eqnarray}
where in the last inequality we have used the Poincar\'{e} inequality 
\[\int^1_0|p|^2dt\leq C\int^1_0|\nabla_t p|^2dt\]
for the function $p(s,t)$ with $p(s,0)=0$ for all $s$. 

Since $\|f\|_{C^2}$ is small enough, one can assume that 
\[C_4\|f\|_{C^2}<\frac{1}{2}.\]
Then we substitute (\ref{e:normest}) into (\ref{e:parint}) and (\ref{e:derp}) to obtain
\begin{eqnarray}
\int^1_0\langle \nabla_s\nabla_s p, p\rangle dt	&\geq& \int^1_0|\nabla_t p|^2 dt-\frac{1}{2}\int^1_0\big(|\nabla_s p|^2+|\nabla_t p|^2\big)dt\notag\\
&=&\frac{1}{2}\int^1_0|\nabla_t p|^2 dt-\frac{1}{2}\int^1_0|\nabla_s p|^2 dt\notag
\end{eqnarray}
Combining this with (\ref{e:2nd}), we obtain the desired inequality. \end{proof}

 \end{document}